\documentclass[11pt]{article}
\usepackage{amsmath, amsfonts, amssymb, amsthm}
\usepackage[dvipsnames]{xcolor}
\usepackage{graphicx} 
\usepackage{a4wide}
\usepackage{hyperref}
\usepackage[title, titletoc]{appendix}
\hypersetup{
    linktocpage=true,
    colorlinks,
    linkcolor={BrickRed},
    citecolor={ForestGreen!90!black},
    urlcolor={Salmon!50!black}
}
\usepackage{tikz-cd}
\usepackage[tight]{minitoc}
\usepackage{tocloft} 
\title{Riemannian geometry of vector and matrix optimal transport}

\author{Boris Khesin\thanks{
        Department of Mathematics,
        University of Toronto, Canada;
        e-mail: \texttt{khesin@math.toronto.edu}, corresponding author}\;
    and Klas Modin\thanks{
        Department of Mathematical Sciences, Chalmers University of Technology and University of Gothenburg, Sweden;
        e-mail: \texttt{klas.modin@chalmers.se}}
    }
\date{}

\theoremstyle{theorem}
\newtheorem{theorem}{Theorem}[section]
\newtheorem{lemma}[theorem]{Lemma}
\newtheorem{proposition}[theorem]{Proposition}

\theoremstyle{definition}
\newtheorem{definition}[theorem]{Definition}
\newtheorem{remark}[theorem]{Remark}
\newtheorem{example}[theorem]{Example}

\newtheorem{problems}[theorem]{Problems}

\numberwithin{equation}{section}

\newcommand{\llangle}{\langle\hspace{-0.5ex}\langle}
\newcommand{\rrangle}{\rangle\hspace{-0.5ex}\rangle}

\begin{document}

\maketitle
\vspace{-5ex}
\begin{abstract}
In this paper we propose a gauge-theoretic approach to the problems of optimal mass transport for vector and matrix densities. This resolves both the issues of positivity and action transitivity constraints.
Bures-type metrics on the corresponding semi-direct product groups of diffeomorphisms and
gauge transformations are related to Wasserstein-type metrics on vector half-densities and matrix densities 
via Riemannian submersions. 
We also describe their relation to Poisson geometry and demonstrate how the  momentum map allows one 
to prove the Riemannian submersion properties. The obtained  geodesic equations turn out to be elaborate vector and matrix versions of Burgers equation.
\end{abstract}


\tableofcontents

\section{Introduction}

The geometric setting of optimal transport relies on the differential geometry of diffeomorphism groups. 
In particular, the problem of moving one mass to another by a diffeomorphism, while minimizing a certain cost, can be understood as a construction of geodesics for an appropriate metric in the space of normalized densities; see, for example, Benamou and Brenier~\cite{BeBr2000}, McCann~\cite{mccann2001polar}, or Villani~\cite{Vi2009} and references therein. Many applications include similar problems where one needs to measure the proximity between different shapes; see, for example, Michor and Mumford~\cite{MiMu2007} or Younes~\cite{Yo2010}. 
However, the diffeomorphism action by itself always preserves the total mass of the density.
One arrives at the problem of constructing a natural extension of the action which would allow one to connect, in the most economical way, densities of different total masses, thus leading to the theory of \emph{unbalanced optimal transport}~\cite{Chizat2015}. 
Furthermore, other applications, for instance to the theory of spin liquids, require distances between densities that are not scalar- but vector- or matrix-valued. The latter problems are the main subject of this work.

There are several obstructions even to properly posing the problem of the mass transport for vector or matrix densities. Namely, a unique diffeomorphism might not allow one to move simultaneously all coordinate densities in a required way. Another difficulty is the positivity constraint for each coordinate density. This constraint is a serious obstacle for many previously suggested scenarios for admissible paths between densities.

In this paper we resolve these difficulties by proposing the {\it gauge-theoretic approach} to the problem. Furthermore, we regard the vector densities as {\it half-densities}, with the ${\rm SO}(k)$-gauge transformations acting on them along with diffeomorphisms, which resolves both the positivity and action transitivity issues. The corresponding group acting on  densities is enlarged from diffeomorphisms to 
a semidirect product group. Such groups differ from one problem to another, and they are described below
in the various cases of balanced and unbalanced vector and matrix densities. 

Next, we propose a tight connection to Poisson geometry and demonstrate how the {\it momentum map} helps compute the horizontal distributions and prove the Riemannian submersion properties.
Finally, it turns out that natural Bures-type metrics on the acting groups and vector/matrix densities
lead to a familiar form of the geodesic equations. Indeed,the geodesic equations become ramifications of the Burgers equations for fully compressible fluids.

\subsection*{Overview of the main results}

Here is a quick overview of the main results, first for the vector half-densities and then for the matrix densities.

\paragraph*{Vector-valued optimal transport} 
To present a sample of the results we start with a Riemannian manifold $M$ of dimension $n$ with volume form $\mu$, and consider the space of normalized vector-valued  half-densities
$$
\mathrm{VHProb} =\{ \mathbf{v} \in C^\infty(M,\mathbb{R}^k ) \otimes \sqrt{\mu}~|~ \int_M \lvert \mathbf{v} \rvert^2 = 1 ~{\rm and ~}
\mathbf{v}\not=0 ~{\rm  on ~} M\}
$$ 
where $k$ and $n$ are unrelated.
These half-densities can be represented as $\mathbf{v}={v}\otimes \sqrt\varrho\in \mathrm{VHProb}$  with $\int_M\varrho=1$, $\varrho > 0$, and $v\in C^\infty(M,\mathbb{S}^{k-1})$. 
The group acting on    half-densities in $\mathrm{VHProb}$ is (the identity component of) the semidirect product 
$\widetilde G=  \operatorname{Diff}(M)\ltimes C^\infty(M,\operatorname{SO}(k))\ni (\varphi, A)$, where
 the diffeomorphism $\varphi$ acts on  half-densities $\mathbf{v}$ by the pullback with its inverse and with the square root of  its Jacobian  appearing as a factor, while the ${\rm SO}(k)$-current $A$ on $M$ is applied point-wise to the obtained vector-valued half-density. 
 This action turns out to be transitive and preserves positivity of the densities, see Section~\ref{sect:group-for-half-densities}. 

One can introduce natural (weak) Riemannnian metrics on both the group $\widetilde G$ and the density space $\mathrm{VHProb}$, and the natural projection becomes a Riemannian submersion (Theorem~\ref{thm:vectorRiemSub}). 
In particular,   the vector Wasserstein distance $\mathit{WV}(\mathbf{v}_0,\mathbf{v}_1)$ between densities $\mathbf{v}_0,\mathbf{v}_1 \in \mathrm{VHProb}$ is
    \[
        \mathit{WV}^2(\mathbf{v}_0,\mathbf{v}_1) = \inf_{u,a,\mathbf{v}}\int_0^1 \int_M \left( \lvert u\rvert^2 +  \operatorname{tr}(aa^\top) \right) \lvert\mathbf{v}\rvert^2 \, dt\,,
    \]
    over time-dependent vector fields $u(t)\in \mathfrak{X}(M)$, vector half-densities $\mathbf{v}(t) \in \mathrm{VHProb}$, and matrix valued functions $a(t) \in C^\infty(M, \mathfrak{so}(k))$ related by the constraints
    \[
        \dot{\mathbf{v}} = -L_u \mathbf{v} + a \mathbf{v} ,\quad \mathbf{v}(0) = \mathbf{v}_0, \quad \mathbf{v}(1) = \mathbf{v}_1,
    \]  
see Definition \ref{def:vectorWass}. Essentially, the distance between any two vector densities
increases whenever more mass has been moved and more mass was necessary to ``rotate'' in the gauge group to relate their vector values.

The corresponding geodesic equations are given by the following:

\begin{theorem}[{\bf = Theorem \ref{prop:governing_abstract}$'$}]
The reduced geodesic equations on the corresponding (co)tangent bundle
$T^{(*)}\widetilde G= T^{(*)}(\operatorname{Diff}(M)\ltimes C^\infty(M,\operatorname{SO}(k)))$, corresponding to an  optimal transport of vector densities,
are
    given by
    \begin{equation*}
        \dot{u} + \nabla_{{u}}{u} = 0, \quad \dot a + \nabla_u a = 0, \quad \dot{\mathbf{v}} + L_{{u}} \mathbf{v} =  a \mathbf{v} 
    \end{equation*}
where $\mathbf{v} = (\varphi,A)\cdot \mathbf{v}_0 = {v}\otimes\sqrt\varrho$.
In particular, the scalar density $\varrho$ defined by $\varrho = \lvert\mathbf{v}\rvert^2$
is transported by $u$: 
$$\dot \varrho +L_u \varrho=0.$$
\end{theorem}

These geodesic equations become natural  extensions of the classical Burgers equation for the velocity $u$ and  transported scalar or gauge field $a$. However, due to complexity of the horizontal distribution in the group $\operatorname{Diff}(M)\ltimes C^\infty(M,\operatorname{SO}(k))$ intertwining different components,  one cannot split out completely the velocity part and describe solutions to these seemingly simple equations solely via  gradient fields as in the classical scalar optimal transport.

Analogous equations describe the unbalanced vector transport, but with the ${\rm Conf}(k)$-gauge group  
of linear conformal transformations replacing $\operatorname{SO}(k)$, see Section \ref{sec:vector-unbalan}. Since we are interested in the local questions and the corresponding geodesic equations, everywhere in the paper we assume that the vector half-densities under consideration are homotopy equivalent and that the corresponding purely topological obstructions for their global gauge equivalence vanish, as we discuss below in Remark \ref{rem:obstruction1}.

  \medskip

\paragraph*{Matrix-valued optimal transport}
Similar metrics can be defined in the matrix density case. Namely, we consider the 
space of normalized matrix densities on $M$ 
$$\mathrm{MProb}:=\{\mathbf S \in \Omega^n(M, {\rm Sym}_+(k))~|~ \, \int_M \operatorname{tr}(\mathbf S)=1
\}$$ 
of 
unit total mass with values in positive-definite quadratic forms of unit trace. 
Here, ${\rm Sym}_+(k)$ denotes the cone of positive-definite quadratic forms of $k$ variables.
The scalar density associated with $\mathbf S$ is  $\operatorname{tr}(\mathbf S)$, and the total mass is normalized: $\int_M \operatorname{tr}(\mathbf S)=1$.

The semidirect product group acting on this space is
$$
\widehat{PG}={\rm Diff}(M)\ltimes C^\infty(M, {\rm PGL}(k, \mathbb R)),
$$ 
with possibly different $k$ and $n$.  The group action is the composition of the diffeomorphism action and a pointwise ``rescaling'' 
as follows: 
For  an element
$$(\varphi, A)\in {\rm Diff}(M)\ltimes C^\infty(M, {\rm PGL}(k, \mathbb R))$$ the action on a matrix density $\mathbf S$ is $(\varphi, A)\cdot \mathbf S = \pi\circ A (\varphi_*\mathbf S) A^\top $,
where $\varphi_*$ acts on each matrix element thought of as a top-form, $A$ acts by pointwise change of basis in the quadratic form, and the pointwise rescaling $\pi$ takes positive-definite quadratic forms to those of unit trace.
Similarly to the above, the  action of the group $\widehat{PG}$  on matrix densities $\mathrm{MProb}$ of unit total masses   is transitive. 

As above, let now $M$ be a Riemannian manifold with metric $(\cdot\,, \cdot)$.
We fix an initial matrix density $\mathbf S_0 \in \mathrm{MProb}$.
A matrix analogue of the Bures-type  $L^2$-metric on the group leads to the following definition of the  matrix Wasserstein distance $\mathit{WM}(\mathbf S_0,\mathbf S_1)$ between two matrix densities $\mathbf S_0,\mathbf S_1 \in \mathrm{MProb}$:
    \[
        \mathit{WM}^2(\mathbf S_0, \mathbf S_1) = \inf_{u,a,\mathbf S}\int_0^1 \int_M \left( \lvert u\rvert^2 +  \operatorname{tr}(aa^\top) \right) \operatorname{tr}(\mathbf S) \, dt\,,
    \]
    where the infimum is taken over time-dependent vector fields $u(t)\in \mathfrak{X}(M)$, matrix densities $\mathbf S(t) \in \mathrm{MProb}$, and matrix valued functions $a(t) \in C^\infty(M, \mathfrak{gl}(k))$ related by the constraints
    \[
        \dot{\mathbf S} = -L_u\mathbf S + a\mathbf S + \mathbf S a^\top ,\quad \mathbf S(0) = \mathbf S_0, \quad \mathbf S(1) = \mathbf S_1, \quad \operatorname{tr}(a\mathbf S)=0.
    \]
Here again, this construction naturally extends the classical Wasserstein distance by penalizing both whenever some mass is to be moved over $M$  and whenever two matrix valued masses need an additional  transformation in the fibers to identify them.
An appropriate metric on the group $\widehat{PG}$ projects to the  above one via Riemannian submersion. 
By  introducing factorization $\mathbf S = S\otimes\varrho$ where $\operatorname{tr}S = 1$ one can write the corresponding reduced geodesic equations on (the cotangent space of) the group $\widehat{PG}$ as follows:

\begin{theorem}[{\bf = Theorem \ref{thm:governing_equations_MDens_norm}$'$}]
    The Hamiltonian equations of motion, corresponding to the reduced geodesic equations on $\widehat{PG}$ for an  optimal transport of matrix densities, are
    \begin{align*}
        & \dot{u} + \nabla_{u}{u} = 0, \qquad ~\qquad [\dot a] + \nabla_u [a] = 0, \\ 
        & \dot{S} + \nabla_u S = a S + Sa^\top , \quad \dot\rho + \operatorname{div}(\rho u) = 0\,,
    \end{align*}
    with the constraint $\operatorname{tr}(aS)\equiv 0$ for selecting $a$ from the coset $[a]\in C^\infty(M,\mathfrak{pgl}(k))$.
\end{theorem}
Note that the geodesic equations turn out to be rather explicit and can be regarded as a Burgers-type intertwined version of compressible chiral fluids. As before, the difficulty of solving them explicitly is in 
the complexity of the horisontal distribution for the corresponding Riemannian submersion.
We discuss the equations in more detail in Section~\ref{sec:matrix-balanced}, as well as an unbalanced version of this transport in Section \ref{sec:matrix-unbalanced}. 
It is worth mentioning that unlike the vector case, all matrix densities on $M$ are homotopy equivalent since the space of
positive-definite quadratic forms is contractible. One might still encounter 
topological obstructions for their gauge equivalence which we assume vanishing, see Remark \ref{rem:obstruction1}.

Analytical results for the corresponding initial value problems are summarized as follows:
\begin{theorem}[{\bf = Theorem (\ref{thm:vect-exist}$+$\ref{thm:matrix-exist})$'$}]
 For smooth initial data $u_0$, $a_0$, and $\mathbf{v}_0$ (or, respectively, $u_0$, $a_0$, $\rho_0 $, and $S_0$)
  there exists a time-interval $[0,T)$ with $T>0$ for which the equations in the above theorems admit unique smooth solutions.
\end{theorem}
These results also admit an extension from the smooth category to the Sobolev $H^s$ setting for $s>n/2+1$, see Remark~\ref{rem:Sobolev}.
\medskip

{\bf Notations} Throughout the paper we adopt the following notations: $\mathrm{Prob}$  and $\mathrm{HProb}$ stand respectively for normalized densities and half-densities, while for non-normalized ones we reserve $\mathrm{Dens}$ and $\mathrm{HDens}$. The vector and matrix objects are marked by V or M, so that, e.g., normalized vector half-densities are $\mathrm{VHProb}$, while all matrix densities are $\mathrm{MDens}$.


\medskip

{\bf Motivation} One of the most straightforward applications of optimal transport of vector densities includes colored medical imaging, for example \emph{multi parametric magnetic resonance imaging} where the color channels correspond to different types of imaging data~\cite{Kather2017}.
By making a colored picture of an organ, one can see ``how far'' it is from its  healthier shape, size, and color.
Medical imaging is also relevant for matrix-valued densities, for example, for \emph{diffusion tensor imaging} where the imaging data consists of point-wise symmetric matrices that describe the diffusion of water in the tissue.

As another application we mentioned the theory of quantum spin liquids in condensed matter physics. It describes a phase of matter that can be formed by interacting quantum spins in certain magnetic materials. This  state is intuitively described as a ``liquid'' of disordered spins, much in the way liquid water is in a disordered state compared to crystalline ice. In this paper we propose a possible measure of the distance between any two such vector-valued states. We discuss several other motivations in the last section.

\medskip
{\bf Other approaches}
There are many approaches to vector and matrix optimal mass transport  dispersed in the literature.
In particular, we mention matrix-valued measures in \cite{Brenier2020}, 
metrics based on a combination of the Lagrangian and intensity modelling in \cite{Thorpe}, the weighted action functional in \cite{Li}, dynamical optimal transport metrics between density matrices in \cite{Carlen}, and a quantum mechanical approach in \cite{ChenMatrix}, although this list, admittedly, only scratches the surface of this broad area.

Possibly the closest to our approach to vector densities is the  multichannel transport in~\cite{Chen2018}, where mass moves both through physical space and between components, with intercomponent exchange prescribed by a graph. The coupling between channels allows one to move beyond solving independent scalar transport problems, and it can be thought of as an analog of the gauge transformations discussed below. In \cite{Ciosmak2021}, the authors developed
a theory for general vector measures, extending Kantorovich--Rubinstein duality. 
That paper showed why scalar transport-ray arguments encounter substantive obstacles in the vector setting. 

The paper \cite{Brenier2020} introduced the Kantorovich--Bures geometry and established metric, topological, and geometric properties of matrix optimal transport, including existence of optimal paths and the relation to scalar unbalanced transport. 
A weighted Wasserstein--Bures distances through a matrix-type continuity equation and convex dynamic action was developed in \cite{Li}, thus providing a framework for studying matrix-valued unbalanced transport. 
A dynamic Benamou--Brenier analogue proposed in \cite{Chen2017} was using noncommutative differential operators. It treats density matrices and spatially distributed matrix densities, linking transport geometry with quantum entropy and control. The static approach of \cite{Peyre2019} to tensor fields demonstrates how quantum relative entropy relaxes marginal constraints, and entropic regularization yields a matrix analogue of Sinkhorn scaling, including barycenter computations. 

It would be interesting to elucidate relations of the above with the setting proposed here.
The main features of the present paper
are 
{\it a)} a gauge-theoretic approach to the optimal transport problems, 
{\it b)} a unified way of dealing with vector and matrix densities due to the common geometric Jacobian-type action on them, 
{\it c)} a treatment of vector densities as half-densities, naturally implying positivity and transitivity of the action, and 
{\it d)} applications of momentum map techniques.
\medskip

{\bf Organization of the paper}
In Section~\ref{sec:classical_OT} we give background on the geometric setting of classical optimal transport.
In Section~\ref{sec:vector_half_densities} we introduce vector half-densities and their corresponding transport maps (based on a semi-direct product), and we give a geometric description of the corresponding transport problem, including the geodesic equations.
In Section~\ref{sect:unbal} we first review the framework of unbalanced optimal transport, and then discuss how to generalize it to vector half-densities.
Finally, in Section~\ref{sec:matrix_densities} we present a universal differential geometric point-of-view, which deals with matrix densities and encapsulates other geometric frameworks. 
Existence results for the geodesic initial value problems are presented in Sections~\ref{sec:existence_vector} and \ref{sec:existence_matrix}, while a more detailed analysis of the corresponding two-point transport problems is an intriguing open question due to its non-convexity.
In Section~\ref{sect:questions}, we discuss open questions and further directions. In the appendices we describe
how the moment map manifests itself in optimal transport (Appendix~\ref{sec:momentum_map_appendix}) and
demostrate the geodesic equations for an alternative metric on matrix densities (Appendix~\ref{sect:alternative}).

\paragraph*{Acknowledgements}
The authors are indebted to Sergei Tabachnikov,  Anton Izosimov and Fran\c{c}ois-Xavier Vialard,  as well as other participants of the workshops
``Math en plein air", for fruitful discussions.
We are also grateful to Conrad Morris for important feedback on Appendix~\ref{sect:alternative}.
B.K. was partially supported by an NSERC Discovery Grant. K.M. was supported by the Swedish Research Council (grant number 2022-03453), the Knut and Alice Wallenberg Foundation (grant number KAW2024.0229), and the Göran Gustafsson Foundation for Research in Natural Sciences and Medicine.
\medskip

\noindent{\bf Data sharing} is not applicable to this article as no datasets were generated or analyzed during the current study.
The authors declare no conflict of interest. 


\section{Optimal transport and semidirect products}\label{sec:classical_OT}

\subsection{Geometric setting of scalar optimal transport}
Let $\mu$ be a smooth reference volume form (or density) 
 of unit total mass on a manifold $M$, 
and consider the projection $\pi : \operatorname{Diff}(M) \to \operatorname{Prob}(M)$ of diffeomorphisms 
onto the space $\operatorname{Prob}(M)$ of smooth probability (i.e. normalized) 
densities on $M$. 
The  diffeomorphism group $\operatorname{Diff}(M)$
is fibered over $\operatorname{Prob}(M)$ by means of this projection
$\pi$ as follows: the fiber over a volume form $\varrho$  consists of all 
diffeomorphisms $\eta$ that push $\mu$ to $\varrho$,  $\eta_*\mu=\varrho$. In other words, 
two diffeomorphisms $\eta_1$ and $\eta_2$ 
belong to the same fiber 
if and only if $\eta_1=\eta_2\circ\varphi$ for some 
diffeomorphism $\varphi$ preserving the volume form $\mu$.
Diffeomorphisms from $\operatorname{Diff}(M)$ act transitively on smooth densities, 
according to the Moser theorem.

Suppose now that $M$ is a compact $n$-dimensional Riemannian manifold with metric $(~,~)$. 
Then one can define (weak) Riemannian metrics on both the diffeomorphism group $\operatorname{Diff}(M)$
and the density space $\operatorname{Prob}(M)$. Namely,  at each point $\eta\in\operatorname{Diff}(M)$
of  the diffeomorphism group is defined in the following straightforward way: 
given $X,Y\in \mathfrak{X}(M)$, the inner product of
two vectors $X\circ\eta,Y\circ\eta\in T_\eta\operatorname{Diff}(M)$  is 
\begin{equation}\label{eq:diff_metric}
\llangle X\circ\eta,Y\circ\eta\rrangle_{\eta}
=\int_M \left(X\circ\eta(x),Y\circ\eta(x)\right)\,\mu(x).
\end{equation}
This metric is right-invariant when restricted to the subgroup
$\operatorname{Diff}_\mu(M)$ of volume-preserving diffeomorphisms, although it is
not right-invariant on the whole group $\operatorname{Diff}(M)$. 
It is worth mentioning that for a   flat manifold $M$ this metric is a  flat
metric on $\operatorname{Diff}(M)$:
a neighborhood of the identity $e\in \operatorname{Diff}(M)$ with the metric 
$\llangle~,~\rrangle_{\operatorname{Diff}} $ is isometric to a neighborhood
in the pre-Hilbert space 
of smooth ``vector-functions'' $\eta: M\to M$ with the $L^2$ inner product
$\langle~,~\rangle_{L^2(M)}$.
\smallskip

On the other hand, on the density space $\operatorname{Prob}(M)$  there also exists a natural metric 
inspired by the following  {\it optimal mass
transport problem}: find a (sufficiently regular) map $\eta:M\to M$ that pushes
the  measure $\mu$ forward to $\varrho$ and attains the minimum
of the $L^2$-cost functional 
$\int_M \operatorname{dist}^2(x,\eta(x))\mu$ among all such maps, where $\operatorname{dist}$ is the Riemannian distance on $M$.
The minimal cost of transport defines the $L^2$  
{\it Wasserstein distance} ${\operatorname{Dist}}$ on densities $\operatorname{Prob}(M)$:
$$
{\operatorname{Dist}}^2(\mu, \varrho)
:=\inf_\eta\Big\{\int_M \operatorname{dist}^2(x,\eta(x))\mu~
|~\eta_*\mu=\varrho\Big\}\,.
$$
One can see that the Wasserstein distance is formally 
generated by the (weak) Riemannian metric on the space $\operatorname{Prob}(M)$ of smooth densities:
\begin{equation}\label{eq:dens_metric}
\langle\dot{\varrho},\dot{\varrho}\rangle_\varrho = 
		\int_M |{\nabla {\theta}}|^2 \varrho \quad\text{for } \dot{\rho} + {\rm div}_\mu({\rho} \nabla {\theta})=0,
\end{equation}
where $\dot{\rho}  \in C^\infty(M)$ is a tangent vector to ${\rm Prob}(M)$ at the point 
$\varrho=\rho\mu$. (The gradient vector field $\nabla{\theta}$ has the minimal $L^2$ energy among all vector fields
moving the volume form $\varrho$ in the direction $\dot{\varrho} $ due to the Hodge decomposition.) 
Thus both $\operatorname{Diff}(M)$ and ${\rm Prob}(M)$ can be regarded as infinite-dimensional Riemannian 
manifolds.

\begin{theorem}[\cite{Otto2001}]\label{rsub}
The bundle map $\pi:\operatorname{Diff}(M)\to {\rm Prob}(M)$ is a
Riemannian submersion of the metric \eqref{eq:diff_metric} 
on the diffeomorphism group 
$\operatorname{Diff}(M)$ to the metric \eqref{eq:dens_metric} 
on the density space ${\rm Prob}(M)$.
The horizontal distribution 
consists of right-translated gradient fields.
\end{theorem}

Recall that for two Riemannian manifolds $P$ and $B$ a  {\it Riemannian submersion}
$\pi: P\to B$ is a mapping onto $B$ that has maximal rank and preserves 
lengths of horizontal tangent vectors to $P$. 
For a bundle $P\to B$ this means that on $P$ there is a distribution 
of horizontal 
spaces that are orthogonal to fibers and projected isometrically to the 
tangent spaces to $B$.


\subsection{Geodesics in the scalar transport}
One of the main problems in optimal mass transport is to describe geodesics on the space $\operatorname{Prob}(M)$ of densities, that is, on the base of the bundle $\operatorname{Diff}(M)\to \operatorname{Prob}(M)$.
However, it is easier to  start instead with the description of geodesics on $\operatorname{Diff}(M)$ and then restrict to horizontal geodesics, which correspond to geodesics on the base.

Define the (inviscid) {\it Burgers equation} \index{Burgers equation} as the 
 evolution equation
\begin{equation}\label{eq:burgers}
\partial_t u+\nabla_u u=0\,
\end{equation}
for a vector field $u$ on $M$, where $\nabla_u u$ stands for the covariant derivative on $M$. 
Consider the flow 
$(t,x)\mapsto \eta(t,x)$ corresponding to  velocity field $u$:
$$
\partial_t \eta (t,x) = u(t,\eta(t,x)), \qquad \eta(0,x)=x. 
$$ 

\begin{proposition}[\cite{Otto2001}]\label{prop:burg}
(1) Solutions of the Burgers equation are time-dependent vector fields on $M$
that describe the following flows of fluid particles: 
each particle moves with constant velocity (defined by the initial condition) 
along a geodesic in $M$.

(2) Geodesics in the group $\operatorname{Diff}(M)$ with respect to 
the above $L^2$-metric  correspond to solutions of the Burgers 
equation. Horizontal geodesics in  $\operatorname{Diff}(M)$ are normal to the submanifold $\operatorname{Diff}_\mu(M)$ and
have potential initial conditions.  
\end{proposition}

\begin{remark}
Recall that the {\it Euler equation} 
$
\partial_t u+\nabla_u u=-\nabla p 
$
on a divergence-free field $u$ on $M$, ${\rm div} \,u=0$,  describes the motion of 
an ideal incompressible fluid filling $M$. It corresponds to 
the equation of the geodesic flow of the right-invariant $L^2$
metric on the group $\operatorname{Diff}_\mu(M)$, regarded as a submanifold $\operatorname{Diff}_\mu(M)\subset \operatorname{Diff}(M)$. Here the term $-\nabla p $ stands for the constraining force orthogonal to $\operatorname{Diff}_\mu(M)$ and keeping the geodesic on this submanifold.
\end{remark}


\subsection{Semi-direct products}
Let $G$ be a Lie group acting from the left on another Lie group $L$ such that the action map $g\cdot L\to L$ is an automorphism: $g\cdot (\ell_1 \ell_2) = (g\cdot \ell_1)(g\cdot \ell_2)$.
Then we can form the semi-direct product group $G\ltimes L$ with product
\begin{equation}
    (g_1,\ell_1)\cdot (g_2, \ell_2) = \big(g_1g_2, \ell_1(g_1\cdot \ell_2)\big).
\end{equation}
The corresponding Lie bracket is
\begin{equation*}
    [(u_1,a_1),(u_2,a_2)] = \big([u_1,u_2]_{\mathfrak{g}},  u_1\cdot a_2 - u_2\cdot a_1 +[a_1,a_2]_{\mathfrak{l}} \big).
\end{equation*}
Assume further that both $G$ and $L$ act from the left on the vector space $V$.
An action of $G\ltimes L$ on $V$ is then given by $(g,\ell)\cdot v = \ell \cdot (g\cdot v)$.
Throughout, the space $V$ takes the role of the elements being transported, so it is essential that this action is transitive.

\begin{example}
The group 
$${\rm Aff}(n) = {\rm GL}(n)\ltimes {\mathbb R}^n
$$ of affine transformations of the $n$-dimensional space is a semidirect product. In this group, an affine transformation is described by a pair $(A,b)$ of a linear operator $A\in {\rm GL}(n)$  and a translation by a vector $b \in {\mathbb R}^n$, so that the group multiplication between two such pairs, 
$$
(A_1, b_1)\circ (A_2, b_2) :=(A_1 A_2, A_1b_2+b_1)\,,
$$ 
is determined by the consecutive application of transformations on elements of ${\mathbb R}^n$.

In particular, the semidirect group $E(3)= {\rm SO}(3)\ltimes {\mathbb R}^3$ of motions of the 3-dimensional space is related to the Kirchhoff equations for the motion of a rigid body in a fluid.
\end{example}

\begin{example}
Similarly, the equations of gas dynamics are related to the semidirect product Lie group 
$$
\widetilde G=\operatorname{Diff}(M)\ltimes C^\infty(M, \mathbb R)\,,
$$ 
which  consists of pairs $(\phi, f) $ with a diffeomorphism $\phi$ of the manifold $M$, and a smooth 
function $f$ on $M$, see e.g. \cite{Khesin2021}. The product in the group $\operatorname{Diff}(M)\ltimes C^\infty(M, \mathbb R)$ 
is given by 
$$(\phi, f) \cdot (\psi, g) := (\phi\circ\psi, \phi_*g+f)\,,
$$ 
where $\phi_*g = g \circ \phi^{-1}$  denotes the pushforward of the function $g$ by the diffeomorphism $\phi$  (or, equivalently, the pullback by $\phi^{-1}$).

The corresponding Lie algebra is the semidirect product 
$\tilde{\mathfrak g}=\mathfrak{X}(M)\ltimes C^\infty(M, \mathbb R)\,,$ which is
$\mathfrak{X}(M)\otimes C^\infty(M, \mathbb R)$ as a vector space with the
commutator  given by
$$
[(u,a), (v,b)]=(L_u v ,-L_u b + L_v a), 
$$
where $L_u$ is the Lie derivative of vector fields or functions on the manifold $M$.

This is the type of semidirect Lie groups and  algebras to be used below, with the group of $\operatorname{Diff}(M)$ acting on the spaces of functions $C^\infty(M, R)$ with values in various vector and matrix spaces $R$.
\end{example}

\begin{example}
    A main example in this paper is $\mathrm{Diff}(M)\ltimes C^\infty(M,\mathrm{SO}(k))$, where $C^\infty(M,\mathrm{SO}(k))$ is equipped with the loop-group structure.
    The product is
    \begin{equation*}
        (\varphi_1,A_1)\cdot (\varphi_2,A_2) = \big( \varphi_1\circ\varphi_2, A_1(A_2\circ\varphi_1^{-1}) \big),
    \end{equation*}
    and the corresponding Lie bracket is
    \begin{equation*}
        [(u, a),(v,b)] = \big( -L_u v, -L_u b + L_v a + [a,b]\big).
    \end{equation*}
    The right-reduced coordinates $(u,a) \equiv (\dot\varphi,\dot A)\cdot (\varphi, A)^{-1}$ are defined by
    \begin{equation*}
        u = \dot\varphi\circ\varphi^{-1}, \quad a = (\dot A + L_u A)A^{-1}.
    \end{equation*}
\end{example}

\medskip

\section{Mass transport of vector half-densities}\label{sec:vector_half_densities}

\subsection{The group action on vector half-densities}\label{sect:group-for-half-densities}
Now we consider the corresponding setup for vector densities and describe a
similar semi-direct product group  acting on them in a transitive way.  

Given a compact orientable manifold $M$ of dimension $n$
define the semi-direct product group $\widetilde G=G\ltimes N$ where 
$G= \operatorname{Diff}(M)$ and $N = C^\infty(M,\operatorname{SO}(k))$.\footnote{One can generalize the setting to non-compact manifolds, e.g.\ $M = \mathbb{R}^n$, by imposing decaying conditions on the objects involved.} 
(Slightly abusing notations, we denote by $\widetilde G$ only the connected component of the identity of this semi-direct product group.)
We define its action on the space 
$$
\mathrm{VHProb} =\{ \mathbf{v} \in C^\infty(M,\mathbb{R}^k ) \otimes \operatorname{HProb}(M)~\Big|~ \int_M \lvert \mathbf{v} \rvert^2 = 1 ~{\rm and ~}
\mathbf{v}\not=0 ~{\rm  on ~} M\}
$$ 
of normalized vector-valued  half-densities.
(Half-densities $\mathrm{HProb}(M)$ are positive top degree forms which under the action of diffeomorphisms are multiplied by the square root of the Jacobian, as we discuss below.)
They can be represented as $\mathbf{v}={v}\otimes \sqrt\varrho\in \mathrm{VHProb}$  with $\int_M\varrho=1$ and $\varrho > 0$.
Here the tensor product is over $C^\infty(M, \mathbb R)$, and one can also assume that 
${v}(x)\in S^{k-1}\subset \mathbb{R}^k $ for all $x\in M$.

The action of $\varphi\in \operatorname{Diff}(M)$ on $\mathbf{v} \in \mathrm{VHProb}$ is the component-wise half-density pushforward:
\[
  \varphi\cdot \mathbf{v} = \varphi_*\mathbf{v} = {v}\circ\varphi^{-1} \sqrt{\varphi_*\varrho}  
\]
The corresponding infinitesimal action of $u\in \mathfrak{X}(M)$ is the component-wise Lie derivative: $u\cdot \mathbf{v} = -L_u \mathbf{v}$.
In particular, relative to the representation $\mathbf{v} = v\otimes \sqrt{\varrho}$, we have
\[
    L_u \mathbf{v} = L_u (v\otimes \sqrt{\varrho}) = (L_u v) \otimes \sqrt{\varrho} + v\otimes \frac{1}{2}\frac{L_u\varrho}{\sqrt{\varrho}}= (L_u v + \frac{L_u\varrho}{2\varrho}v )\otimes\sqrt{\varrho}.
\]
The action of the gauge transformation $A \in N$ on $\mathbf{v}\in \mathrm{VHProb}$ is pointwise: 
$$(A\cdot \mathbf{v})(x) = A(x)\mathbf{v}(x).$$
The corresponding infinitesimal action is $(a\cdot \mathbf{v})(x) = a(x)\mathbf{v}(x)$ for $a \in T_e N = C^\infty(M,\mathfrak{so}(k))$.

\begin{definition}\label{def:action_vect}
The combined group action for 
$\widetilde G=  \operatorname{Diff}(M)\ltimes C^\infty(M,\operatorname{SO}(k))$ on  half-densities
in $\mathrm{VHProb}$ is as follows. 
The action of an element $(\varphi, A)\in \widetilde G=  \operatorname{Diff}(M)\ltimes C^\infty(M,\operatorname{SO}(k))$ is
$$
(\varphi, A):\, \mathbf{v}={v}\otimes \sqrt\varrho 
\mapsto 
A(\varphi_*\mathbf{v}) = A(v\circ\varphi^{-1})\sqrt{\varphi_*\varrho}\, ,
$$
where the latter expression at a point $x\in M$, relative to a reference density $\mu$, is 
$$
\Big(A(v\circ\varphi^{-1})\sqrt{\varphi_*\varrho}\Big)(x)
=A(x)v(\varphi^{-1}(x)) \sqrt{\operatorname{Jac}_\mu(\varphi^{-1})(x)}\,\rho(\varphi^{-1}(x))\mu(\varphi^{-1}(x))\,.
$$ 
In other words, the diffeomorphism acts on  half-densities by the pullback with its inverse and with the square root of  the Jacobian  appearing as a factor, while the ${\rm SO}(k)$-current $A$ on $M$ is applied point-wise to the obtained vector-valued half-density. 
\end{definition} 

The combined infinitesimal action of $(u,a) \in T_{(id,e)}(G\ltimes N)$ is 
\begin{equation}\label{eq:inf_action_on_v}
    (u,a)\cdot \mathbf{v} = -L_u \mathbf{v} + a \mathbf{v} .
\end{equation}

\begin{lemma}\label{lem:trans}
    The action of the group $\widetilde G=  \operatorname{Diff}(M)\ltimes C^\infty(M,\operatorname{SO}(k))$
    on the half-densities $\mathrm{VHProb}$ is transitive. 
\end{lemma}
\begin{proof}
Indeed, by Moser's theorem, the action of a diffeomorphsim on the densities $\varrho$ is transitive, as all such densities in  $\operatorname{Prob}(M)$  have the same total volume. This means that any vector half-density
$\mathbf{v}={v}\otimes \sqrt\varrho$ can be mapped to any other while matching the density
$\varrho$ and obtaining another vector current $\tilde{v}(x)$ of the same pointwise length as ${v}(x)$.
The currents $A\in C^\infty(M, {\rm SO}(k))$, being pointwise orthogonal transformations, do not change the vector length and act transitively on vectors of the same length at each point $x\in M$, which completes the proof. (Here we use the assumption on the absence of purely topological obstructions, see Remark \ref{rem:obstruction1}, so the local transitivity in the bundle extends to a global one.) 
\end{proof}

Note that the fibers of the corresponding fibration $\widetilde G\to \mathrm{VHProb}$ depend on the choice of the reference density $\mathbf{v}\in \mathrm{VHProb}$ and are given by its stabilizer. For instance, for a vector density 
$\mathbf{v}={v}\otimes \sqrt\varrho$ whose $v$-components are all the same in a certain coordinate system in $\mathbb{R}^k$, its stabilizer is the subgroup $\mathrm{Stab}_\mathbf{v}:=\operatorname{Diff}_\mu(M)  \ltimes C^\infty(M, \operatorname{SO}(k-1))$. Indeed, it consists of those diffeomorphisms $\varphi$ which preserve  density $\varrho=\lvert\mathbf{v}\rvert^2$ and at each point of $x\in M$ the subgroup ${\rm SO}(k-1)\subset {\rm SO}(k)$ fixes a given nonzero vector 
${v}(x)$.

\medskip

Consider now the principal bundle $\pi\colon G\ltimes N \to \mathrm{VHProb}$ given by the action on a fixed element $\mathbf{v}_0\in \mathrm{VHProb}$, i.e., $\pi(\varphi,A) = A\, \varphi_* \mathbf{v}_0$.
This fixed element corresponds, as we shall see, to the initial vector half-density in the transport problem.

\begin{lemma}\label{lem:div_free_vertical}
    The vertical distribution at $(\varphi,A)\in G\ltimes N$ is given by
    \begin{equation}
      V_{(\varphi,A)}(G\ltimes N) =     \{(u,a)\cdot (\varphi,A)\mid -L_u \mathbf{v} + a \mathbf{v} = 0, \; \mathbf{v} = A\, \varphi_* \mathbf{v}_0 \}.
    \end{equation}
    In particular, if $(u,a)\cdot (\varphi,A)$ is vertical, then
    \begin{equation}
        L_u \lvert\mathbf{v}\rvert^2 = 0,
    \end{equation}
    i.e., $u$ is divergence-free with respect to the density $\lvert\mathbf{v}\rvert^2$.
\end{lemma}

\begin{proof}
    The first part follows since the projection $\pi$ is by left action on $\mathrm{VHProb}$.
    For the second part, notice  that if $(u,a)$ is vertical one has
    \begin{equation*}
        L_u \lvert\mathbf{v}\rvert^2 = 2 \mathbf{v}\cdot L_u \mathbf{v}  = \mathbf{v}\cdot (L_u \mathbf{v} - a \mathbf{v} ) = 0\,,
    \end{equation*}
    since $a$ is skew-symmetric.
\end{proof}
\smallskip


\subsection{Riemannian metrics and submersion}
Now we define a Riemannian metric on $\widetilde G= G\ltimes N$.

\begin{figure}
    \includegraphics{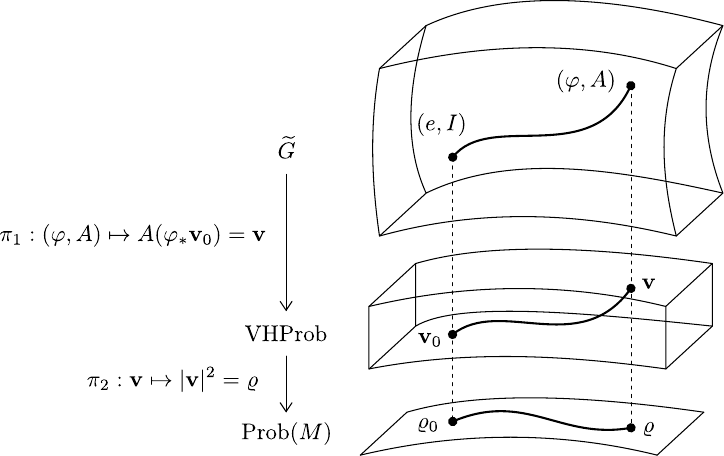}
    \caption{Illustration of the relation between the group, the vector half-densities, and the densities. The projections, first by the action of $\widetilde G$ on $\mathbf v_0$ and then by taking the squared norm of the vector half-density, give rise to Riemannian submersions.}
    \label{fig:enter-label}
\end{figure}
 
\begin{definition}\label{def:Bures}
The {\it Bures-type $L^2$-metric} on the group $\widetilde G=  \operatorname{Diff}(M)\ltimes C^\infty(M,\operatorname{SO}(k))$ 
associated with the fixed initial vector half-density $\mathbf{v}_0$ is given by 
\begin{equation}\label{eq:Bures_metric_vectors}
    \langle (u,a)\cdot (\varphi,A),(u,a)\cdot (\varphi,A)\rangle_{(\varphi,A)}
    = \int_M \left( \lvert u\rvert^2 +  \operatorname{tr}(aa^\top)  \right) \varphi_* \lvert \mathbf{v}_0 \rvert^2 \,.
\end{equation}
\end{definition}

This is a natural combination of the non-invariant $L^2$ metric on $\mathrm{Diff}(M)$ and the Bures metric on matrices.
The first term of this metric evaluates how much mass has been moved, while the second term measures how much
mass was necessary to rotate in the gauge group to relate their vector values.\footnote{One can also introduce a family of Bures-type metrics with different weighting of the these terms for specific applications, but it can be absorbed by rescaling of the second term and is omited below.}

\begin{proposition}\label{prop:nondegenerate}
The  Bures-type $L^2$-metric on the group $\widetilde G=  \operatorname{Diff}(M)\ltimes C^\infty(M,\operatorname{SO}(k))$ is nondegenerate.
\end{proposition}

\begin{proof} If two group elements move the reference vector density to two densities which differ by the underlying scalar one, the distance between them is positive due to the first term in the distance function. Indeed, the first term in the distance \eqref{eq:Bures_metric_vectors} corresponds to  the standard $L^2$ metric on diffeomorphisms, and  the above metric projects to the classical Wasserstein metric on scalar densities, while the latter  is known to be nondegenerate. 

Now assume that two different group elements belong to the same fiber, and hence differ at some point on $M$ by a nontrivial gauge transformation. Due to the smooth depence on the point, there is some neighborhood 
on $M$ where they differ. Now the nondegeneracy of the metric \eqref{eq:Bures_metric_vectors} along the fiber, i.e. in its second term, follows from that in the finite-dimensional setting, the nondegeneracy of the metric on $\operatorname{SO}(k)$ induced by the Killing form.
\end{proof}


Once the metric on the group  is fixed, one can give a dynamical formulation of the corresponding optimal transport problem on the space of vector densities.

\begin{definition}\label{def:vectorWass}
    The \emph{vector Wasserstein distance} $\mathit{WV}(\mathbf{v}_0,\mathbf{v}_1)$ between densities $\mathbf{v}_0,\mathbf{v}_1 \in \mathrm{VHProb}$ is given by the following formula:
    \[
        \mathit{WV}^2(\mathbf{v}_0,\mathbf{v}_1) = \inf_{u,a,\mathbf{v}}\int_0^1 \int_M \left( \lvert u\rvert^2 +  \operatorname{tr}(aa^\top) \right) \lvert\mathbf{v}\rvert^2 \, dt\,,
    \]
    over time-dependent vector fields $u(t)\in \mathfrak{X}(M)$, vector half-densities $\mathbf{v}(t) \in \mathrm{VHProb}$, and matrix valued functions $a(t) \in C^\infty(M, \mathfrak{so}(k))$ related by the constraints
    \[
        \dot{\mathbf{v}} = -L_u \mathbf{v} + a \mathbf{v} ,\quad \mathbf{v}(0) = \mathbf{v}_0, \quad \mathbf{v}(1) = \mathbf{v}_1.
    \]  
\end{definition}

\medskip

The corresponding horizontal distribution on the group is described as follows.

\begin{proposition}\label{prop:horizontal}
    The horizontal distribution at $(\varphi,A)\in G\ltimes N=\operatorname{Diff}(M)\ltimes C^\infty(M,\operatorname{SO}(k))$ is given by
    \begin{multline*}
        H_{(\varphi,A)}(G\ltimes N) = \\
        \left\{
            \left(
                \langle \nabla\theta,{v}\rangle 
                - \langle \nabla{v},\theta\rangle
                ,
                \theta\otimes{v}-{v}\otimes\theta
            \right) \mid {\it for ~all~} \,\theta\in C^\infty(M,\mathbb{R}^k)
        \right\},
    \end{multline*}
    where $\mathbf{v} = (\varphi,A)\cdot \mathbf{v}_0 = {v}\otimes\sqrt\varrho$.
\end{proposition}

Note that this description of the horisontal distribution is much less explicit than 
the one for the group $\operatorname{Diff}(M)$ with the $L^2$-metric fibered over scalar densities: in that case $H_{\varphi}(G)$ is given by gradient vector fields $\{\nabla\theta \mid {\it for ~all~} \,\theta\in C^\infty(M,\mathbb{R})\}$ and it does not depend on the base point.  In the semidirect product setting above, the horisontal distribution can be understood as a continuous combination of gradients and, moreover, it depends on the base point via the direction vector $v$, but not on the density part $\varrho$.
\medskip

We are proving this proposition below, but now note that such a description allows one to define a metric on vector densities  inherited from the metric on the group 
via a Riemannian submersion. As its corollary we obtain the following key statement.

\begin{theorem}\label{thm:vectorRiemSub}
    The metric defined on the group
    $\widetilde G=G\ltimes N$ descends to the space of vector half-densities $\mathrm{VHProb}$ 
    as a Riemannian submersion, where, for $\dot{\mathbf{v}}\in T_\mathbf{v} \mathrm{VHProb}$ with $\mathbf{v} = {v}\otimes\sqrt\varrho$, it is given by
    \[
        \langle \dot{\mathbf{v}}, \dot{\mathbf{v}}\rangle_{v} = \int_M \left( \lvert \langle \underbrace{\nabla\theta,{v}\rangle - \langle \nabla{v},\theta\rangle}_{u}\rvert^2 +  \lvert \theta\rvert^2 \right)\varrho, 
    \]
    where $\theta\in C^\infty(M,\mathbb{R}^k)$ is (implicitly) defined by
    \[
        \dot{\mathbf{v}} = -L_{u}\mathbf{v} + a \mathbf{v} 
    \]
    with $u$ and $a$ given, respectively, by $u=\langle\nabla\theta,{v}\rangle
     - \langle \nabla{v},\theta\rangle$ and $a=\theta\otimes{v}-{v}\otimes\theta$.    
\end{theorem}

\begin{proof}
    Let $(\eta,B) \in \widetilde G$ be an element of the isotropy group for $\mathbf{v}\in \mathrm{VHProb}$.
    Then, from Lemma~\ref{lem:div_free_vertical} it follows that $\eta$ is volume preserving with respect to the density $\lvert\mathbf{v}\rvert^2$, i.e., $\eta_*\lvert\mathbf{v}\rvert^2 = \lvert\mathbf{v}\rvert^2$.
    Then from the formula~\eqref{eq:Bures_metric_vectors} it  follows that the metric descends to the quotient and therefore corresponds to a Riemannian submersion.
    From Proposition~\ref{prop:horizontal} we then get an explicit form of the descending metric.
\end{proof}


A proof of Proposition \ref{prop:horizontal} on 
the description of horizontal spaces, implying the Riemannian submersion property, can be obtained from the following explicit form of the momentum map for that group.

\begin{lemma}\label{lem:momentum_maps}
    The momentum map $\Phi\colon T^* \mathrm{VHProb}\to \big(\mathfrak{X}(M)\ltimes C^\infty(M,\mathfrak{so}(k))\big)^*$ for the Hamiltonian action of $\widetilde G=G\ltimes N$ on $T^*\mathrm{VHProb}$, under the natural $L^2$-pairing of $\mathrm{VHProb}$ with itself, is given by
    \begin{equation}\label{eq:half}
        \Phi({v}\otimes \sqrt\varrho, \theta\otimes\sqrt\varrho)
        =  \Bigg( \frac{1}{2}\Big(\langle d\theta,{v}\rangle - \langle d{v},\theta\rangle \Big)\otimes\varrho,  \frac{1}{2}\Big(\theta\otimes{v}-{v}\otimes\theta\Big)\otimes\varrho\Bigg).
    \end{equation}
\end{lemma}

\begin{proof}
    Recall that elements $\mathbf{v}\in \mathrm{VHProb}$ can be expressed as a vector current  ${v}\otimes \sqrt{\varrho}$ for ${v}\colon M\to S^{k-1}$ and normalized density $\varrho\in\operatorname{Prob}(M)$.
    Correspondingly, we write $\boldsymbol{\theta} \in T_\mathbf{v}^*\mathrm{VHProb}$ as $\boldsymbol{\theta} = \theta\otimes\sqrt\varrho$ for some function $\theta\colon M\to \mathbb{R}^k$.
    Take now $(u,a) \in \mathfrak{g}\ltimes\mathfrak{n}$.
    Recall that the cotangent lifted action is always Hamiltonian, i.e., given by Hamiltonian vector fields (cf.~Appendix~\ref{sec:momentum_map_appendix}).
    The momentum map is then defined by
    \[
        \langle (u,a), \Phi(\mathbf{v},\boldsymbol{\theta})\rangle = \langle (u,a)\cdot \mathbf{v}, \boldsymbol{\theta} \rangle \qquad
        \forall \, (u,a) \in \mathfrak{g}\ltimes \mathfrak{n}.
    \]
    Since $(u,a)\cdot \mathbf{v} = -L_u ({v}\otimes \sqrt\varrho) + a {v}\otimes\sqrt\varrho$, we get
    \begin{multline*}
        \langle (u,a)\cdot \mathbf{v}, \boldsymbol{\theta}\rangle = \int_M \left( -L_u{v} + \frac{-L_u\varrho}{2\varrho} + a{v} \right)\cdot \theta\; \varrho =  \\
        \int_M \langle u, -\frac{1}{2}(d {v}\cdot \theta - d\theta\cdot {v})\rangle \, \varrho - \int_M \operatorname{tr}\big(a(\theta\otimes{v})\big)\, \varrho .
    \end{multline*}
    Since 
    \begin{equation*}
        -\operatorname{tr}(a(\theta\otimes v)) = \frac{1}{2}\operatorname{tr}(a (\theta\otimes v - v\otimes \theta)^\top)    
    \end{equation*}
    the result now follows.
\end{proof}

{\it Proof of Proposition \ref{prop:horizontal}} now immediately follows from the above description of the momentum map $\Phi$  since, by construction, the horizontal distribution $H_{(\varphi,A)}(G\ltimes N)$ 
is given by the condition
    \[         \langle (u,a), \Phi(\mathbf{v},\boldsymbol{\theta}) \rangle = 0     \]
    if and only if $(u,a)\cdot (\varphi,A)$ is vertical. \hfill$\square$




\subsection{Geodesics and Hamiltonian equations}

We start by deriving the coadjoint operator for the Lie algebra  $\tilde{\mathfrak{g}}=\mathfrak{g}\ltimes\mathfrak{n}$
of the group $\widetilde G=  \operatorname{Diff}(M)\ltimes C^\infty(M,\operatorname{SO}(k))$.

\begin{lemma}\label{lem:ad_star}
    The $\operatorname{ad}^*$ map for the Lie algebra $\tilde{\mathfrak{g}}=\mathfrak{g}\ltimes\mathfrak{n}$ is given by
    \begin{equation*}
        \operatorname{ad}^*_{(u,a)}(m\otimes\varrho, \beta\otimes\varrho) = \left(
            -L_u (m\otimes\varrho) + \operatorname{tr}(\beta d a)\otimes\varrho,
            -L_u (\beta\otimes\varrho) -[a,\beta]\otimes\varrho
         \right)
    \end{equation*}
    where $(m\otimes\varrho, \beta\otimes\varrho)\in \tilde{\mathfrak{g}}^*\simeq\mathfrak{g}^*\oplus\mathfrak{n}^*$
    and $m\in     \Omega^1(M)$, $\beta \in  C^\infty(M, {\rm SO}(k))$, while $\varrho\in  \Omega^n(M) $, and  the tensor products are over $C^\infty(M,\mathbb R)$.
\end{lemma}

\begin{proof}
    We have that $\operatorname{ad}_{(u_1,a_1)}(u_2,a_2) = (-L_{u_1} u_2, -L_{u_1}a_2 + L_{u_2}a_1 + [a_1, a_2])$.
    Thus, for variables $(m\otimes\varrho, \beta\otimes\varrho)$ dual to the Lie algebra,
    \begin{align*}
        &\langle \operatorname{ad}_{(u_1,a_1)}^*(m\otimes\varrho, \beta\otimes\varrho), (u_2,a_2)\rangle
        = \langle -L_{u_1}u_2, m\otimes\varrho \rangle + \langle -L_{u_1} a_2 + L_{u_2}a_1 + [a_1,a_2], \beta\otimes\varrho \rangle = \\
        & \langle u_2, L_{u_1}(m\otimes \varrho)\rangle 
        + \langle a_2, L_{u_1}(\beta\otimes\varrho)\rangle 
        + \langle \iota_{u_2} da_1, \beta\otimes\varrho\rangle +\langle [a_1,a_2],\beta\otimes\varrho\rangle = \\
        & \langle u_2, L_{u_1}(m\otimes \varrho) - \operatorname{tr}(\beta da_1)\otimes\varrho\rangle 
        + \langle a_2, L_{u_1}(\beta\otimes\varrho) -[a_1,\beta]\otimes\varrho\rangle \,,
    \end{align*}
    which implies the result.
\end{proof}

The geodesic equations for the Bures-type metric in Definition~\ref{def:Bures}, as a Hamiltonian system on the Lie--Poisson space $\tilde{\mathfrak{g}}^*\simeq\mathfrak{g}^*\oplus\mathfrak{n}^*$, is given by the following theorem.

\begin{theorem}\label{prop:governing_abstract}
    Consider the Hamiltonian 
    on $T^*\widetilde G= T^*(\operatorname{Diff}(M)\ltimes C^\infty(M,\operatorname{SO}(k)))$
    of the form
    \begin{equation*}
        H(u^\flat\otimes\varrho, a\otimes\varrho, \mathbf{v}) = \frac{1}{2}\int_M \frac{\lvert u^\flat\otimes \varrho \rvert^2}{\lvert\mathbf{v}\rvert^2} + \frac{1}{2}\int_M \frac{\varrho}{\lvert\mathbf{v}\rvert^2}\operatorname{tr}\left(aa^\top\right) \varrho  .
    \end{equation*}
    corresponding to the kinetic energy Lagrangian on $T\widetilde G$ and 
    given in reduced coordinates $u = \dot\varphi\circ\varphi^{-1}$, $\varrho = \varphi_*|\mathbf{v}_0|^2$, $\mathbf{v} = A\varphi_*\mathbf{v}_0$, and $a = (\dot A + L_u A) A^{-1}$.
    Then the Hamiltonian equations of motion corresponding to the geodesic equations on $\widetilde G$ for the  optimal transport on vector densities are
    \begin{equation}\label{eq:geodesic_eq_vector_balanced}
        \dot{u} + \nabla_{{u}}{u} = 0 , \quad \dot a + \nabla_u a = 0, \quad \dot{\mathbf{v}} + L_{{u}} \mathbf{v} =  a \mathbf{v} .
    \end{equation}
In particular, the scalar density $\varrho$ defined by $\varrho = \lvert\mathbf{v}\rvert^2$
is transported by $u$: 
$$\dot \varrho +L_u \varrho=0.$$
\end{theorem}

\begin{remark}
The geodesic equations in Theorem \ref{prop:governing_abstract} are natural vector extensions of the classical Burgers equation $\dot{u} + \nabla_{{u}}{u} = 0$ and the passive scalar transport $ \dot a + \nabla_u a = 0$. Indeed, the $L^2$-metric of Definition~\ref{def:Bures}
does not depend on the orientation of the vectors $v$, so the geodesic path in the $\operatorname{Diff}(M)$-part of the group should also be independent of that orientation, and hence it does not affect the evolution of $u$. Note however, that the horizontal distribution $H_{(\varphi,A)}(G\ltimes N)$ in the group $\operatorname{Diff}(M)\ltimes C^\infty(M,\operatorname{SO}(k))$ found in
Proposition~\ref{prop:horizontal} is rather complicated, which does not allow one to describe the geodesics on vector densities, i.e. solutions of those equations,
as explicitly as for scalar ones as the flows of gradient fields.

Note also that the equations in Theorem~\ref{prop:governing_abstract} are on $\tilde{\mathfrak{g}}^*\times \mathrm{VHProb}$, which is a reduced form of $T^*\widetilde G=\tilde{\mathfrak{g}}^*\times (\operatorname{Diff}(M)\ltimes C^\infty(M,\mathbb{R}^k))$ taking into account the reduction of base isometries.
\end{remark}





\begin{proof} 
In the normalization $\varrho = \lvert\mathbf{v}\rvert^2$ one has
$$\dot \varrho = 2 \langle \mathbf{v}, \dot{\mathbf{v}}\rangle = - 2\langle \mathbf{v},L_u \mathbf{v} \rangle = - L_u \lvert\mathbf{v}\rvert^2 = -L_u \varrho,$$
which implies that $\varrho$ is transported by $u$.

    The variational derivatives of the Hamiltonian are
    \begin{equation*}
        \frac{\delta H}{\delta (u^\flat\otimes\varrho)} = u, \quad 
        \frac{\delta H}{\delta (a\otimes\varrho)} =  a, \quad 
        \frac{\delta H}{\delta \mathbf{v}} = -\lvert u \rvert^2 \mathbf{v} -  \operatorname{tr}(a a^\top) \mathbf{v}\,.
    \end{equation*}
    Indeed, the explicit form of these variational derivatives is obtained by rewriting the Hamiltonian as
    \begin{equation*}
        H = \frac{1}{2}\langle u^\flat\otimes\varrho, u \rangle + \frac{1}{2}\langle a\otimes\varrho, a \rangle 
        = \frac{1}{2}\langle u^\flat\otimes\lvert\mathbf{v}\rvert^2, u \rangle + \frac{1}{2}\langle a\otimes \lvert\mathbf{v}\rvert^2, a \rangle .
    \end{equation*}

    Next, the general form of the corresponding Hamiltonian equations is obtained in Proposition~\ref{prop:poisson_reduction_abstract} in Appendix, namely
    \begin{equation}\label{eq:ham_eq_abstract}
        \frac{d}{dt}(u^\flat\otimes\varrho, a\otimes\varrho) + \operatorname{ad}^*_{(\frac{\delta H}{\delta u^\flat\otimes\varrho},\frac{\delta H}{\delta a\otimes\varrho})}(u^\flat\otimes\varrho, a\otimes\varrho)= - \Phi(\mathbf{v}, \frac{\delta H}{\delta \mathbf{v}}).
    \end{equation}
    Then from Lemma~\ref{lem:momentum_maps} we have
    \begin{align*}
        &\Phi(\mathbf{v}, \frac{\delta H}{\delta \mathbf{v}}) =
        \Phi\left( v\otimes\sqrt{\varrho}, -(|u|^2  {v}+ \operatorname{tr}(aa^\top)  {v}) \otimes\sqrt{\varrho}\right) = \\
        &\left(-\frac{1}{2}\left(\langle d (|u|^2  v+ \operatorname{tr}(aa^\top)  v) ,{v}\rangle - \langle d v, (|u|^2  v+ \operatorname{tr}(aa^\top)  v) \rangle \right) \otimes\varrho, \right. \\ 
        &\left. -\frac{1}{2} \left((|u|^2  v+ \operatorname{tr}(aa^\top)  v)\otimes v -  v\otimes (|u|^2  v+ \operatorname{tr}(aa^\top)  v) \right) \right) \\
        &= \left(-\frac{1}{2} d \left(|u|^2 + \operatorname{tr}(aa^\top)\right)\otimes\varrho , 0\right),
    \end{align*}
    and from Lemma~\ref{lem:ad_star}
    \begin{align*}
        &\operatorname{ad}^*_{(\frac{\delta H}{\delta u^\flat\otimes \varrho}, \frac{\delta H}{\delta a\otimes\varrho})}(u^\flat\otimes\varrho,a\otimes\varrho) =  \big( L_u (u^\flat\otimes\varrho) - \operatorname{tr}(a \, d a), L_u (a\otimes\varrho) - [a,a]\otimes\varrho \big) = \\
        &  \big( L_u (u^\flat\otimes\varrho) +\frac{1}{2} d \operatorname{tr}(aa^\top), L_u (a\otimes\varrho) \big).
    \end{align*}
    Since $\varrho = \lvert\mathbf{v}\rvert^2$ is transported by $u$, it decouples from the equation, so from equation~\eqref{eq:ham_eq_abstract} we get
    \begin{equation*}
        \dot u^\flat = \underbrace{- L_u u^\flat  - \frac{1}{2} d \operatorname{tr}(aa^\top)}_{\text{from }\operatorname{ad}^*}  \underbrace{+\frac{1}{2}d|u|^2 + \frac{1}{2}d\operatorname{tr}(aa^\top)}_{\text{from }-\Phi} = -L_u u^\flat - \frac{1}{2}d\lvert u\rvert^2 .
    \end{equation*}
    Thus, the equation for the one-form $u^\flat$ is
    \begin{equation*}
        \dot u^\flat  + \underbrace{L_u u^\flat - \frac{1}{2}d \lvert u \rvert^2}_{(\nabla_u u)^\flat} = 0, 
    \end{equation*}
    which yields the equation for $u$ as
    \begin{equation*}
        \dot u + \nabla_u u = 0.
    \end{equation*}
    For the evolution of $a\otimes\varrho$ we get
    \begin{equation*}
        \frac{d}{dt}a\otimes\varrho = - L_u (a\otimes\varrho)
    \end{equation*}
    which decouples as
    \begin{equation*}
        \dot a + \nabla_u a = 0 .
    \end{equation*}
    The evolution for $v$ is obtained from the infinitesimal action formula \eqref{eq:inf_action_on_v}. 
\end{proof}


\begin{problems} 
The above consideration was related to a particular Bures-type $L^2$-metric. It would be interesting  to see if other metrics give natural
equations of geodesics or the corresponding transport problem. In particular, do the $\dot H^1$- and the Fisher--Rao metrics manifest themselves in the vector setting?
\end{problems}

\begin{remark} \label{rem:obstruction1}
To spell out the assumptions on homotopy equivalence of vector half-densities, note that 
by considering the normalization $v\in S^{k-1}$ the question of fixing a connected component 
for ${\rm VHProb}$ boils down to the homotopy classes of maps $M^n \to S^{k-1}$. For $n < k-1$, all maps are homotopically trivial, and hence one has a unique connected component of vector half-densities on $M$. When $n=k-1$ the connected components are distinguished by the degree of the map $M^{n}\to S^n$. If $n > k-1$ the description of such connected components is, in general, an open question. For example, let $M=S^3 $ and $k=3$, the connected components of vector half-densities are 
enumerated by $\pi_3(S^2) = \mathbb Z$, the Hopf invariant. And, more generally, if $M$ is a sphere, one deals with the homotopy groups of spheres that, in general, are still not known.

An extension of the map between vector functions (or half-densities) to their ${\rm SO}(k)$-gauge equivalence is a more subtle question. Now one needs to lift vector half-densities from ${\rm VHProb}$, which 
are  $S^{k-1}$-sections over $M$, to gauge transformations between them, which are  ${\rm SO}(k)$-sections over $M$.
The latter can be regarded as sections of the ${\rm SO}(k-1)$-bundle over ${\rm VHProb}$. The corresponding theory implies that  obstructions to such an extension belong to the groups $H^j(M, \pi_{j-1}(F))$ for $F={\rm SO}(k-1)$. In the paper, as we study local questions and geodesic equations, we assume that all those obstructions are trivial, so that the extensions are topologically possible. 

Note also that for matrix densities discussed in detail below, the first extension holds automatically. Indeed there is just one connected component in ${\rm MProb}$, i.e. all maps $M\to {\rm Sym}_+(k)$
are homotopically equivalent, since positive-definite quadratic forms form a contractible set. 
On the other hand, the second part of extending identification of matrix densities 
to a gauge equivalence is based on essentially the same obstruction theory as above, since the stabilizer of a positive quadratic form in $k$ variables is isomorphic to  ${\rm SO}(k)$. 
 \end{remark}

\begin{remark}\label{remark4}
In the past a typical approach to define optimal transport on vector densities has been as follows (see e.g. \cite{Chen2018}): one considers an $k$-tuple (or equivalently, a vector) of densities
$(\varrho_1, ..., \varrho_k)$ on $M$ and an $k$-tuple of diffeomorphisms $(\varphi_1, ..., \varphi_k)$ acting on them component-wise. In addition, one introduces a matrix of allowed density exchanges between the components, while
making sure that the densities after the exchanges do not lose their positivity. The corresponding distance between 
$k$-tuples  of densities defined as the cost of such an exchange is, generally speaking, only a divergence: it is not symmetric and requires an additional symmetrization. 

Apparently, this setting might turn out to be related to the flows of $k$-phase fluids, where the densities
and flows are introduced  separately for each phase but   related by a common constraint, such as the total incompressibility, see \cite{Izosimov2023}. Symmetries of multiphase fluids are described by the corresponding {\it multiphase diffeomorphism groupoid}. The groupoid approach also allows one to introduce metrics and distances on the space of 
 $k$-tuples  of densities, which have the fluid dynamical background. It would be interesting to find a precise relation between multiphase fluid groupoids and more general optimal transport on vector densities.
 
 One should also note that in the  approach of the present paper via the semidirect group, where a common diffeomorphism is 
 acting on vector half-densities, the exchanges between components are achieved by gauge transformations.
 It requires neither additional positivity adjustment (as it is automatically achieved by squaring half-densities to obtain the densities themselves) nor additional symmetrization (as it follows from the invariance of the  metric introduced
 on the group and the Riemannian submersion). 
 
 Futhermore, geodesics in the semidirect group setting are  naturally related to Burgers-type fluid equations, while the approach itself easily extends to other cases, such as, e.g., matrix-valued densities, as we explain in Section \ref{sec:matrix_densities}.
\end{remark}

\subsection{Existence of solutions}\label{sec:existence_vector}

\begin{theorem}\label{thm:vect-exist}
    For initial data $u_0\in \mathfrak{X}(M)$, $a_0\in C^\infty(M,\mathfrak{so}(k))$, and $\mathbf{v}_0 \in C^\infty(M,\mathbb{R}^k)\otimes\sqrt{\mu}$, there exists a time-interval $[0,T)$ with $T>0$ for which the equations \eqref{eq:geodesic_eq_vector_balanced} admits a unique smooth solution $u\in C^\infty(M\times [0,T), TM)$, $a\in C^\infty(M\times [0,T),\mathfrak{so}(k))$, and $\mathbf{v} \in C^\infty(M\times [0,T), \mathbb{R}^k\otimes\sqrt{\mu})$.
\end{theorem}

\begin{proof}
    The equation for $u$ is Burger's classical equation, completely decoupled from the other equations.
    From the method of characteristic lines we obtain a unique smooth solution $u(x,t)$ on a maximal time interval $[0,T)$ (just before a shock emerges). 
    By the same method, we obtain a corresponding time-dependent diffeomorphism $\varphi(x,t)$.
    Since $a$ is infinitesimally advected by $u$, and since $\varphi$ integrates $u$, any solution $a(x,t)$ fulfills $a(\cdot,t) = a_0\circ\varphi(\cdot,t)^{-1}$.
    Since $\varphi$ is unique and a diffeomorphism, $a(x,t)$ is unique and smooth.
    Next, for each $x\in M$ one can solve the non-autonomous matrix flow equation
    \begin{equation*}
        \frac{d}{dt}\tilde A = (a\circ\varphi)\tilde A, \quad A(x,0) = I.
    \end{equation*}
    Since the solution depends smoothly on the coefficients, this gives a smooth matrix field $\tilde A(x,t)\in \mathrm{SO}(k)$ for $t\in [0,T)$.
    Now set $A(\cdot,t) := \tilde A(\varphi(\cdot,t)^{-1}, t)$.
    It then follows that $\frac{d}{dt}\tilde A = (\dot A + L_u A)\circ\varphi$.
    Consequently, $A(x,t)$ fulfills the equation
    \begin{equation*}
        \dot A + L_u A = a A
    \end{equation*}
    which is exactly the right-reconstruction equation, namely $(\dot\varphi,\dot A)\circ (\varphi,A)^{-1} = (u,a)$.
    Therefore, we obtain the final part of the solution by $\mathbf{v}(\cdot,t) = A(\cdot,t)\varphi(\cdot,t)_*\mathbf{v}_0$.
    Since both $\varphi$ and $A$ are smooth and unique, the solution is smooth and unique.
\end{proof}




\section{Unbalanced mass transport} \label{sect:unbal}

\subsection{Various settings of the unbalanced transport}
Unbalanced mass transport considers geodesics between densities of different total masses. Let ${\rm Dens}(M)=\Omega^n(M, {\mathbb R}_+)$ stand for all densities  with arbitrary total masses on an $n$-dimensional manifold $M$. 

It became natural now to associate an unbalanced transport with consideration of the semidirect product group 
$$
 G=\operatorname{Diff}(M)\ltimes C^\infty(M, \mathbb R_+),
$$ 
see \cite{Gallouet2021}, Section 3.2.1.
Then $(\varphi, \lambda)\in \operatorname{Diff}(M)\ltimes C^\infty(M, \mathbb R_+)$ acts on a density $\mu \in {\rm Dens}(M)$ as follows:
$$
\mu\mapsto (\varphi^*\mu)\cdot \lambda^2.
$$
This action is  transitive: one can get an arbitrary density from any initial one by choosing an appropriate factor $\lambda$.
(Here and in \cite{Gallouet2021} one uses the multiplicative group of nonvanishing functions, but there is also an additive version: $(\varphi, f)\in \operatorname{Diff}(M)\ltimes C^\infty(M, \mathbb R),$
where $\lambda:=\exp(f)$ and the action is $\mu\mapsto \varphi^*\mu\cdot \exp(2f)$.) 

One introduces a metric on the group $ G$, which projects (as a Riemannian submersion) to a metric on ${\rm Dens}(M)$, see details in \cite{Gallouet2021}. The obtained metric on 
${\rm Dens} (M)= \operatorname{Prob}(M) \times \mathbb R_+$ is a cone over the metric on the space of normalized densities $\operatorname{Prob}(M)$ on $M$ and quadratic scaling in $\mathbb R_+$, see formula (3.22) in \cite{Gallouet2021} (see also earlier papers \cite{Kondratyev2015, Chizat2015}). 

\begin{remark}
   In \cite{Khesin2024} a simple unbalanced optimal transport was described, where one uses much smaller conic group $G=\operatorname{Diff}(M)\times \mathbb R_+$, which projects with a Riemannian submersion to the conic extension of densities $\operatorname{Prob}(M) \times \mathbb R_+$. However, to see the similarity with vector densities described above one needs the ``large" group $ G=\operatorname{Diff}(M)\ltimes C^\infty(M, \mathbb R_+)$. 
\end{remark}


\subsection{Vector unbalanced transport}\label{sec:vector-unbalan}

By replacing ${\rm SO}(k)$ with the conformal linear group 
${\rm Conf}(k):={\rm SO}(k)\times {\mathbb R}_+$ (rotations and dilations) as the main gauge group,  one can consider an {\it unbalanced vector optimal transport}. 
The standard unbalanced transport becomes a particular case of this more universal framework for $k=1$.

Namely, in the unbalanced setting we 
define the semi-direct product group ${\bf G}=  \operatorname{Diff}(M) \ltimes C^\infty(M, {\rm Conf}(k))$.
It acts on the space of all (non-normalized) vector-valued  half-densities
$$
\mathrm{VHDens} =\{ \mathbf{v} \in C^\infty(M,\mathbb{R}^k ) \otimes \operatorname{HDens}(M)~|~ 
\mathbf{v}\not=0 ~{\rm  on ~} M\}
$$ 
We still represent those densities as $\mathbf{v}={v}\otimes \sqrt\varrho\in \mathrm{VHDens} $ with $v(x)\in S^{k-1}$ in a unique way.
The combined group action is given by the same formula as before.

\begin{definition} 
The combined group action for 
${\bf G}=  \operatorname{Diff}(M)\ltimes C^\infty(M,\operatorname{Conf}(k))$ on  half-densities
in $\mathrm{VHDens} $ is as follows:
for $(\varphi, A)\in  \operatorname{Diff}(M)\ltimes C^\infty(M,\operatorname{Conf}(k))$
$$
(\varphi, A):\, \mathbf{v}={v}\otimes \sqrt\varrho \mapsto 
(A{v})\circ \varphi^{-1} 
\cdot\sqrt{ \varphi_*\varrho}\,.
$$
\end{definition} 

In the same way one can see that the action of the group ${\bf G}=  \operatorname{Diff}(M)\ltimes C^\infty(M,\operatorname{Conf}(k))$ half-densities in $\mathrm{VHDens} $  is transitive. 

The corresponding combined infinitesimal action of $(u,a) \in T_{(id,e)}\bf G$ is given by the same formula
\begin{equation} 
    (u,a)\cdot \mathbf{v} = -L_u \mathbf{v} + a \mathbf{v} \,,
\end{equation}
where $a\in C^\infty(M, \mathfrak{conf}(k))$, i.e. $a$ has the scalar diagonal part, since the Lie algebra 
$\mathfrak{conf}(k)$ is a direct product: $\mathfrak{conf}(k)=\mathfrak{so}(k)\oplus I \mathbb R$, where $I$ is the $k\times k$ identity matrix.

All the above considerations, including the geodesic equations in Theorem~\ref{prop:governing_abstract},
have the same form, where the matrix term $a$, which before was skew-symmetric and hence zeros on the diagonal, may now have a scalar diagonal part:
    \begin{equation*}
        \dot{u} + \nabla_{{u}}{u} = 0 , \quad \dot a + \nabla_u a = 0, \quad \dot{\mathbf{v}} + L_{{u}} \mathbf{v} =  a \mathbf{v} .
    \end{equation*}
Here, however,  the scalar density $\varrho = \lvert\mathbf{v}\rvert^2$
is not transported by $u$ any longer, but changes as in the scalar unbalanced transport:
$$\dot \varrho +L_u \varrho= -\frac2k\operatorname{tr}( a) \varrho$$
due to the present diagonal part of $a$.

\begin{remark}
    The existence and uniqueness theorem for vector unbalanced transport mimics Theorem \ref{thm:vect-exist} of the balanced case.
\end{remark}


\section{Gauge theory behind mass transport of matrix densities} \label{sec:matrix_densities}

The setting of matrix-valued densities
requires a different acting group.
Let $M$ be a manifold of dimension $n$, and fix a (normalized) reference density $\mu$.

It turns out that, unlike the vector case, it is more natural to start with the unbalanced transport of matrix densities.

\subsection{Unbalanced optimal transport of matrix-valued densities}\label{sec:matrix-unbalanced}

\begin{definition} \label{def:MDens}
    The space of matrix densities is the space ${\rm MDens}:=\Omega^n(M, {\rm Sym}_+(k))$ of (non-normalized) densities with values in positive-definite quadratic forms in ${\mathbb R}^k$.
\end{definition}

Consider now (the connected component of the identity for) the  semidirect product group 
$$
\widehat G={\rm Diff}(M)\ltimes C^\infty(M, {\rm GL}(k, \mathbb R)),
$$ 
 where $k$ and $n$ could be different. This group acts on the spaces of  matrix densities
${\rm MDens}=\Omega^n(M, {\rm Sym}_+(k))$ as follows. 
For  
$$(\varphi, A)\in {\rm Diff}(M)\ltimes C^\infty(M, {\rm GL}(k, \mathbb R))$$ the action on a matrix density 
$$\mathbf S \in \Omega^n(M, {\rm Sym}_+(k))\subset  \Omega^n(M) \otimes {\rm Sym}(k)$$ is
$$
(\varphi, A)\cdot \mathbf S =  A (\varphi_*\mathbf S) A^\top \,,
$$
where $\varphi_*$ acts on each matrix element thought of as a top-form.
The density associated with $\mathbf S$ is thus $\operatorname{tr}(\mathbf S)$, so the total mass is
$$
{\rm Mass}(\mathbf S):=\int_M \operatorname{tr}\mathbf S.
$$

\begin{remark}\label{rem:closure}
Note that Definition~\ref{def:MDens} is consistent with the definition of vector half-densities: if $\mathbf S:=\mathbf v\otimes \mathbf v$, a matrix density of rank 1 (and hence, being {\it only positive semi-definite}, it belongs to the {\it closure of} ${\rm Sym}_+(k)$), then the vector half-density $\mathbf v$ gives rise to a matrix density $\mathbf S$. Furthermore, the action of the group $\widetilde G$ on vector half-density $\mathbf v$
is consistent with the action of the group $\widehat G$ on the corresponding matrix density $\mathbf S=\mathbf v\otimes \mathbf v$. 
See Figure~\ref{fig:universal_diagram} for an illustration.
\end{remark}

\begin{lemma}
The  action of the group $\widehat G= {\rm Diff}(M) \ltimes C^\infty(M, {\rm GL}(k, \mathbb R))$  on
matrix densities ${\rm MDens}$ of arbitrary total masses is transitive. 
\end{lemma}

\begin{proof} The proof is similar to that of Lemma \ref{lem:trans} above. 
Transitivity on  densities of fixed total mass is provided
by Moser's theorem. On the other hand,  given two matrix densities with the same scalar density part 
the statement follows from transitivity of the pointwise ${\rm GL}(k)$-action on positive-definite quadratic forms: any such form is reduced to the sum of squares in an appropriate basis.
(Here again we use the assumption on the absence of purely topological obstructions, see Remark \ref{rem:obstruction1}.) 
\end{proof}

\medskip

\begin{remark}
It is worth noticing that such a group action naturally leads to an {\it unbalanced} version of the matrix
optimal transport. 
Preservation of the total mass ${\rm Mass}(\mathbf S)=\int_M \operatorname{tr}(\mathbf S)$ 
would require the action of currents preserving the trace of matrices. 
This could be achieved by the action of ${\rm SO}(k)$-currents (i.e. the group $\widetilde G$) on matrix densities, but their action would not be transitive. 
Indeed,  the action of such a group $\widetilde G$ (instead of $\widehat G$, but on matrix densities rather than on vector ones) would preserve not only the trace, but the whole pointwise spectrum of $\mathbf S$, 
and hence would reduce  a positive-definite 
quadratic form to the sum of squares with positive coefficients, where the coefficients, 
being the quadratic form's eigenvalues, remain invariants of the density under this action.
The forms with different eigenvalues (depending on a point in $M$) prevent transitivity of the group $\widetilde G$-action. Also note that the stabilizer subgroup for a matrix-valued density $\mu\otimes \Lambda\in {\rm MDens}$ again depends on the choice of this reference density, since for a quadratic form $\Lambda\in  {\rm Sym}_+(k)$ it is the group ${\rm SO}_\Lambda(k)$, the special orthogonal group for the inner product defined by $\Lambda$.
\end{remark}

\medskip

Lift the action of the group $\widehat G$ from the space of densities $\mathrm{MDens}$ to the Hamiltonian action on its cotangent bundle $T^*\mathrm{MDens}$. We identify $T^*_{\mathbf S}\mathrm{MDens}$ with $C^\infty(M, \mathfrak{gl}(k)/\mathfrak{so}(k))$ via the point-wise Frobenius inner product. For an element $P\in C^\infty(M, \mathfrak{gl}(k))$ consider its coset  $[P]\in C^\infty(M, \mathfrak{gl}(k)/\mathfrak{so}(k))$ defined  by pointwise addition of an arbitrary skew-symmetric matrix.

\begin{lemma}\label{lem:momentum_map_matrix_case}
    Let $(\mathbf S,[P])\in T^*\mathrm{MDens}$.
    Then the momentum map $$\Phi\colon T^* \mathrm{MDens}\to \big(\mathfrak{X}(M)\ltimes C^\infty(M,\mathfrak{gl}(k))\big)^*$$ for the action of $\widehat G$ on $T^*\mathrm{MDens}$ is given by
    \begin{equation}
        \Phi(\mathbf S, [P])
        = \Big( \operatorname{tr}(\mathbf S dP), (P+P^\top)\mathbf S \Big).
    \end{equation}
\end{lemma}

\begin{proof}
    The infinitesimal action of $(u,a)\in \mathfrak{X}(M)\ltimes C^\infty(M,\mathfrak{gl}(k))$ on $\mathbf S\in \mathrm{MDens}$ is
    \[
        (u,a)\cdot \mathbf S = - L_u \mathbf S + a\mathbf S + \mathbf Sa^\top .
    \]
    By definition, the momentum map fulfills
    \[
        \langle \Phi(\mathbf S, [P]), (u,a)\rangle = \langle [P], (u,a)\cdot \mathbf S \rangle
        = \int_M \operatorname{tr}\Big(P\big(-L_u \mathbf S + a \mathbf S + \mathbf Sa^\top\big) \Big),
    \]
    which gives 
    \begin{align*}
        &= \int_M \operatorname{tr}\Big(P\big(
            -d\iota_u \mathbf S + (a \mathbf S + \mathbf Sa^\top)
        \big)\Big)
        \\ &= \int_M \iota_u \operatorname{tr}(\mathbf S dP)
        + \int_M \operatorname{tr}\big( (P^\top+P) \mathbf S a^\top\big)
        \\ &= \Big\langle \big(\operatorname{tr}(\mathbf S d P), (P^\top+P)\mathbf S \big), (u,a)\Big\rangle .
    \end{align*}
The latter expression of the momentum map implies the result. 
\end{proof}

As above, let now $M$ be a Riemannian manifold with metric $(\cdot\,, \cdot)$.
We fix an initial matrix density $\mathbf S_0 \in \mathrm{MDens}$.

\begin{definition}\label{def:MBures}
The {\it Bures-type  $L^2$-metric} on the group $\widehat G=  \operatorname{Diff}(M)\ltimes C^\infty(M,\operatorname{GL}(k))$
is given by 
\begin{equation}\label{eq:metric_Bures_MDens}
    \langle (u,a)\cdot (\varphi,A),(u,a)\cdot (\varphi,A)\rangle_{(\varphi,A)}
    = \int_M \left( \lvert u\rvert^2 + \operatorname{tr}(aa^\top)  \right) \operatorname{tr}(A (\varphi_*\mathbf S_0 ) A^\top) \,.
\end{equation}
\end{definition}

The consideration mimicking Proposition \ref{prop:nondegenerate}
shows that this metric is nondegenerate. 

\begin{definition}
    The \emph{matrix Wasserstein distance} $\mathit{WM}(\mathbf S_0,\mathbf S_1)$ between two matrix densities $\mathbf S_0,\mathbf S_1 \in \mathrm{MDens}$ is given by
    \[
        \mathit{WM}^2(\mathbf S_0, \mathbf S_1) = \inf_{u,a,\mathbf S}\int_0^1 \int_M \left( \lvert u\rvert^2 +  \operatorname{tr}(aa^\top) \right) \operatorname{tr}(\mathbf S) \, dt\,,
    \]
    over time-dependent vector fields $u(t)\in \mathfrak{X}(M)$, matrix densities $\mathbf S(t) \in \mathrm{MDens}$, and matrix valued functions $a(t) \in C^\infty(M, \mathfrak{gl}(k))$ related by the constraints
    \[
        \dot{\mathbf S} = -L_u\mathbf S + a\mathbf S + \mathbf S a^\top ,\quad \mathbf S(0) = \mathbf S_0, \quad \mathbf S(1) = \mathbf S_1.
    \]
\end{definition}



\medskip

As before, these two metrics are defined as natural extensions of the classical Wasserstein metric: For any path between two matrix densities the first term penalizes if much mass is to be moved over $M$, while the second evaluates how much two matrix valued masses differ by measuring how far one needs to rotate the coordinates to identify them. 

For convenience, we now introduce the factorization $\mathbf S = S\otimes\varrho$ where $\operatorname{tr}S = 1$ (this corresponds to the factorization $\mathbf v = v\otimes \sqrt{\varrho}$ of the vector half-density above).
In this factorization, we have that 
\begin{equation}\label{eq:rho_evolution_MDens}
    \dot\varrho = \frac{d}{dt}\operatorname{tr}(\mathbf S) = - L_u\varrho + 2\operatorname{tr}(a S)\varrho
\end{equation}
and
\[
    \dot S = -L_u S + a S + S a^\top - 2 S \operatorname{tr}(a S).
\]

\begin{theorem}
    Consider the projection $\pi\colon \widehat G \to \mathrm{MDens}$ given by the action $\pi(\varphi,A) = A(\varphi_* \mathbf{S}_0)A^\top$.
The projection $\pi$ is a Riemannian submersion of two Riemannian manifolds: the group $\widehat  G=  \operatorname{Diff}(M)\ltimes C^\infty(M,\operatorname{GL}(k))$ equipped with the metric~\eqref{eq:metric_Bures_MDens}   descends to the space of matrix densities $\mathrm{MDens}$ equipped with the Riemannian metric corresponding to the distance function $WM$.

Namely, for $\mathbf{S} = S\otimes \varrho$ and $\dot{\mathbf S}\in T_{\mathbf S} \mathrm{MDens}$, 
the Riemannian  metric on $\mathrm{MDens}$ for the  distance function $WM$ is given by
    \begin{equation}\label{eq:metric_MDens}
        \langle \dot{\mathbf{S}}, \dot{\mathbf{S}}\rangle_{\mathbf S} = \int_M \Big(  | \underbrace{\operatorname{tr}(S\nabla P)}_{u}|^2 +  \operatorname{tr}\big( \underbrace{(P+P^\top)S}_{a} \underbrace{S(P+P^\top)}_{a^\top} \big) \Big)\varrho, 
    \end{equation}
    where $[P]\in T_{\mathbf S}^*\mathrm{MDens}$ is (implicitly) defined by
    \[
        \dot{\mathbf S} = -L_u\mathbf S + a\mathbf S + \mathbf S a^\top
    \]
    with $u$ and $a$ as given in equation~\eqref{eq:metric_MDens}.   
\end{theorem}

\begin{proof}
    The statement that the metric~\eqref{eq:metric_Bures_MDens} decends to the base $\mathrm{MDens}$ 
    is implied by the fact that its dependence  on the group element $(\varphi,A)$ only comes via the projection to the base point $\mathbf S$.
    The Legendre transformation for the metric~\eqref{eq:metric_Bures_MDens} is given by the  inertia operator 
    \begin{equation}\label{eq:legendre_MDens}
    \widehat{\mathfrak{g}} \ni (u,a) \mapsto (u^\flat\otimes\varrho,  a \otimes \varrho)\in \widehat{\mathfrak{g}}^*.
    \end{equation}
    The formula \eqref{eq:metric_MDens} for the metric  then follows from the momentum map in Lemma~\ref{lem:momentum_map_matrix_case} and the definition of the metric~\eqref{eq:metric_Bures_MDens}.
\end{proof}

\medskip

We are now in a position to study geodesics.
To this end, consider the quadratic Hamiltonian on $\widehat{\mathfrak{g}}^*\times \mathrm{MDens}$ obtained via the Legendre transform~\eqref{eq:legendre_MDens}:

\begin{equation}\label{eq:Hamiltonian_MDens}
    \widehat{H}(u^\flat\otimes\varrho, a\otimes\varrho, S\otimes\varrho) = \frac{1}{2}\int_M \Big( |u|^2 + \operatorname{tr}(aa^\top) \Big)\varrho    
\end{equation}

\begin{theorem}\label{thm:governing_equations_MDens}   
    The Hamiltonian equations of motion, corresponding to the reduced geodesic equations on $\widehat{G}$ for an unbalanced optimal transport of matrix densities, are
    \begin{equation}\label{eq:geodesic_eq_matrix_unbalanced}
        \begin{aligned}
            & \dot{u} + \nabla_{u}{u} = 2\operatorname{tr}(a S)u , \quad \quad \dot a + \nabla_u a = (|u|^2 + \operatorname{tr}(aa^\top))S+ 2\operatorname{tr}(a S)a , \\ & \dot\rho + \operatorname{div}(\rho u) = 2\operatorname{tr}(aS)\rho, \quad 
            \dot{S} + \nabla_u S = a S + Sa^\top - 2 \operatorname{tr}(aS) S.
        \end{aligned}
    \end{equation}
\end{theorem}

Note that if $a$ is skew-symmetric then all the terms with $\operatorname{tr}(aS)$ vanish (as well as the term $(|u|^2 +  \operatorname{tr}(aa^\top))S$), and thus the equations simplify noticeably.
We will also observe similar simplifications for the balanced version of matrix transport in the next section. 
Again, the equations are on $\hat{\mathfrak{g}}^*\times \mathrm{MDens}$, which is a reduced form of $T^*\hat G=\hat{\mathfrak{g}}^*\times (\operatorname{Diff}(M)\ltimes C^\infty(M,\mathrm{Sym}_+(k)))$ taking into account the reduction of base isometries.

\begin{proof}
    The variational derivatives of the Hamiltonian are
    \begin{equation*}
        \frac{\delta \widehat H}{\delta (u^\flat\otimes\varrho)} = u, \quad 
        \frac{\delta \widehat H}{\delta (a\otimes\varrho)} =  a, \quad 
        \frac{\delta \widehat H}{\delta \mathbf{S}} = -\frac{1}{2}\big( \lvert u \rvert^2 +  \operatorname{tr}(aa^\top)\big) [I] \,,
    \end{equation*}
    where $[I] \in T^*_{\mathbf{S}}\mathrm{MDens}$ is the co-set of the identity matrix $I$. 

    Next, the general form of the corresponding Hamiltonian equations is obtained in Proposition~\ref{prop:poisson_reduction_abstract} in Appendix, namely
    \begin{equation}\label{eq:ham_eq_abstract_MDens}
        \frac{d}{dt}(u^\flat\otimes\varrho, a\otimes\varrho) = \operatorname{ad}^*_{(\frac{\delta H}{\delta u^\flat\otimes\varrho},\frac{\delta H}{\delta a\otimes\varrho})}(u^\flat\otimes\varrho, a\otimes\varrho) - \Phi(\mathbf{S}, \frac{\delta H}{\delta \mathbf{S}}).
    \end{equation}
    Then from Lemma~\ref{lem:momentum_map_matrix_case} we have
    \begin{align*}
        \Phi(\mathbf{S}, \frac{\delta H}{\delta \mathbf{S}})& =
        \Big( -\frac{1}{2}\operatorname{tr}\Big(\mathbf S d \big( \lvert u \rvert^2 + \operatorname{tr}(aa^\top)\big) \Big), -\big( \lvert u \rvert^2 +  \operatorname{tr}(aa^\top)\big)\mathbf S \Big) \\ 
        &= \left(-\frac{1}{2} d \left(|u|^2 + \operatorname{tr}(aa^\top)\right)\otimes\varrho , -\left(|u|^2 + \operatorname{tr}(aa^\top)\right)\mathbf{S}\right),
    \end{align*}
    and a calculation as in Lemma~\ref{lem:ad_star} shows that 
    \begin{equation*}
        \operatorname{ad}^*_{(u,a)}(m\otimes\varrho, \beta\otimes\varrho) = \Big(
            -L_u (m\otimes\varrho) - \operatorname{tr}(\beta d a^\top)\otimes\varrho,
            -L_u (\beta\otimes\varrho)
         \Big).
    \end{equation*}
    Thus,
    \begin{align*}
        \operatorname{ad}^*_{(\frac{\delta H}{\delta u^\flat\otimes \varrho}, \frac{\delta H}{\delta a\otimes\varrho})}(u^\flat\otimes\varrho,a\otimes\varrho) &=  \big( -L_u (u^\flat\otimes\varrho) - \operatorname{tr}(a \, d a^\top), -L_u (a\otimes\varrho) \big) \\
        &  =\big( -L_u (u^\flat\otimes\varrho) -\frac{1}{2} d \operatorname{tr}(aa^\top), -L_u (a\otimes\varrho) \big).
    \end{align*}
    From equation \eqref{eq:ham_eq_abstract_MDens} we then get
    \begin{align*}
        \frac{d}{dt}(u^\flat\otimes\varrho) &= -L_u(u^\flat\otimes\varrho) + \frac{1}{2}d|u|^2 \otimes\varrho  \\
        &= -(\nabla_u u)^\flat \otimes\varrho - u^\flat\otimes L_u\varrho  .
    \end{align*}
    Combined with equation \eqref{eq:rho_evolution_MDens} this yields
    \begin{align*}
        \dot u + \nabla_u u =  2\operatorname{tr}(a S)u.
    \end{align*}
    The term $\operatorname{tr}(a S)$ can be regarded as a Lagrange multiplier for the constraint $\operatorname{tr}(S)=1$ related to our choice of normalization of $\mathbf{S}$.
    
    Likewise, for the $a\otimes\varrho$ component we get
    \begin{align*}
        \frac{d}{dt}(a\otimes\varrho) = -L_u(a\otimes\varrho) + (|u|^2 +  \operatorname{tr}(aa^\top))S\otimes\varrho,
    \end{align*}
    which yields
    \begin{align*}
        \dot a + \nabla_u a = 2\operatorname{tr}(a S)a + (|u|^2 +  \operatorname{tr}(aa^\top))S.
    \end{align*}
    Finally, the evolutions for $S$ and $\varrho$ are obtained from the infinitesimal action formula as derived earlier. 
\end{proof}

\subsection{Existence of solutions}\label{sec:existence_matrix}

\begin{theorem}\label{thm:matrix-exist}
    For initial data $u_0\in \mathfrak{X}(M)$, $a_0\in C^\infty(M,\mathfrak{gl}(k)$, $\rho_0 \in C^\infty(M)$, and $S_0\in C^\infty(M,\mathrm{Sym}_+(k))$ with $\operatorname{tr}(S_0(x))=1$ for all $x\in M$, there exists a time-interval $[0,T)$ with $T>0$ for which the equations \eqref{eq:geodesic_eq_matrix_unbalanced} admits a unique smooth solution $u\in C^\infty(M\times [0,T), TM)$, $a\in C^\infty(M\times [0,T),\mathfrak{gl}(k))$, $\rho \in C^\infty(M\times [0,T))$, and $S\in C^\infty(M\times [0,T),\mathrm{Sym}_+(k))$.
\end{theorem}

\begin{proof}
    First notice that the equations for $u$, $a$, and $S$ are independent of $\rho$.
    The strategy is to first solve for $u$, $a$, and $S$.
    The trick is to rewrite the equations in the variables $\dot\varphi = u\circ\varphi$, $\tilde a = a\circ\varphi$, and $\tilde S = S\circ\varphi$, where $\varphi(x,0) = x$ so that the new variables have the same initial conditions as the old ones.
    Because each left hand side in \eqref{eq:geodesic_eq_matrix_unbalanced} is the advection derivative $\frac{d}{dt} + \nabla_u$, composition with $\varphi$ gives
    \begin{align*}
        &\nabla_{\dot\varphi}\dot\varphi = 2 \operatorname{tr}(\tilde a \tilde S)\dot\varphi, \quad \frac{d}{dt}\tilde a = (\lvert \dot\varphi \rvert^2 + \operatorname{tr}(\tilde a \tilde a^\top)\tilde S + 2\operatorname{tr}(\tilde a \tilde S)\tilde a \\
        & \frac{d}{dt}\tilde S = \tilde a \tilde S + \tilde S \tilde a^\top - 2\operatorname{tr}(\tilde a\tilde S)\tilde S.
    \end{align*}
    For each $x\in M$, these equations constitute a finite-dimensional smooth system of ordinary differential equations. 
    Given the initial conditions, each such system has a unique local solution.
    Since $M$ is compact, the infimum of the existence times over all $x\in M$ is strictly positive.
    Since the solution depends smoothly on the initial data, we have that $\dot\varphi(x,t)$, $\tilde a(x,t)$, and $\tilde S(x,t)$ depends smoothly on both $x$ and $t$.
    Now construct the corresponding path of diffeomorphisms by solving for each $x\in M$ the non-autonomous equations
    \begin{equation*}
        \frac{d}{dt}\varphi(x,t) = \dot\varphi(x,t), \quad \varphi(x,0) = x.
    \end{equation*}
    Again, $\varphi(x,t)$ is smooth in both $x$ and $t$.
    Further, there exists a strictly positive time $T>0$ such that $\varphi$ remains a diffeomorphisms for $t\in [0,T)$.
    We now obtain the solution as
    \begin{align*}
        &u(\cdot,t) = \dot\varphi(\varphi^{-1}(\cdot,t), t), \quad a(\cdot,t) = \tilde a(\varphi^{-1}(\cdot,t), t), \\
        &S(\cdot,t) = \tilde S(\varphi^{-1}(\cdot,t), t), \quad \rho(\cdot,t)\mu = \varphi(\cdot,t)_*(\rho_0\mu).
    \end{align*}
    Since $\varphi$ is a diffeomorphisms on $[0,T)$, the solution is smooth and unique for $t\in [0,T)$.  
\end{proof}

\begin{remark}\label{rem:Sobolev}
    The existence result can be extended from the smooth category to the Sobolev $H^s$ category, as long as $s>n/2+1$.
    Indeed, this follows 
    since the right-hand side for the equations expressed in the variables $\dot\varphi, \tilde a, \tilde S$ only uses operations that are smooth in the $H^s$ category (cf.~\cite{EbMa1970}).    
\end{remark}



\subsection{Balanced matrix transport}\label{sec:matrix-balanced}

While the matrix version intrinsically favors the unbalanced version, one can develop the balanced matrix transport by normalizing the matrix-valued densities and changing the gauge group from $\mathrm{GL}(k)$-currents to $\mathrm{PGL}(k)$-currents.
\medskip

\begin{definition} The space of normalized matrix densities is 
$$\mathrm{MProb}:=\{\mathbf S \in \Omega^n(M, {\rm Sym}_+(k))~|~ \, \int_M \operatorname{tr}\mathbf S=1.
\}$$ 
In other words, the space of densities with values in positive-definite quadratic forms on ${\mathbb R}^k$ of unit trace.
\end{definition}

The semidirect product group acting on this space is
$$
\widehat{PG}={\rm Diff}(M)\ltimes C^\infty(M, {\rm PGL}(k, \mathbb R)).
$$ 
The group action is the composition of the previously defined (unbalanced) action and a pointwise rescaling
as follows: 
For  an element
$$(\varphi, A)\in {\rm Diff}(M)\ltimes C^\infty(M, {\rm PGL}(k, \mathbb R))$$ the action on a matrix density $\mathbf S$ is
$$
(\varphi, A)\cdot \mathbf S = \pi\circ A (\varphi_*\mathbf S) A^\top \,,
$$
where $\varphi_*$ acts on each matrix element thought of as a top-form, $A$ acts by pointwise change of basis in the quadratic form, and the pointwise rescaling $\pi$ takes positive-definite quadratic forms to those of unit trace.
As before, the  action of the group $\widehat{PG}= {\rm Diff}(M) \ltimes C^\infty(M, {\rm PGL}(k, \mathbb R))$  on
matrix densities $\mathrm{MProb}$ of unit total masses is transitive. 
\smallskip

The unbalanced results in Section~\ref{sec:matrix-unbalanced} are readily adapted to the balanced case. 
In particular, the geodesic equations, given by
Theorem~\ref{thm:governing_equations_MDens} in the unbalanced case, now look as follows:

\begin{theorem}\label{thm:governing_equations_MDens_norm}   
    The Hamiltonian equations of motion, corresponding to the reduced geodesic equations on $\widehat{PG}$ for a balanced optimal transport of matrix densities, are
    \begin{align*}
        & \dot{u} + \nabla_{u}{u} = 0, \quad \quad \dot a + \nabla_u a = \lambda I, \quad \operatorname{tr}(aS) = 0, \\ 
        & \dot\rho + \operatorname{div}(\rho u) = 0, \quad \dot{S} + \nabla_u S = a S + Sa^\top \,,
    \end{align*}
    where $\lambda \in \mathbb{R}$ is a Lagrange multiplier for the constraint $\operatorname{tr}(aS)=0$.
\end{theorem}

\begin{proof}
    The Lie algebra element $[a]\in C^\infty(M,\mathfrak{pgl}(k))$ is the coset $[a] = a + \mathbb{R}I$ for $a\in C^\infty(M,\mathfrak{gl}(k))$.
    The action on $\mathbf S$ of a coset $[a]$ with representative $a'\in [a]$ is given by $(a'+cI) \mathbf S + \mathbf S(a'+cI)^\top$, where $c$ is determined by the condition $\operatorname{tr}((a'+cI) \mathbf S + \mathbf S(a'+cI)^\top) = 0$.
    Thus, we can uniquely choose $a := a'+cI \in [a]$ such that $\operatorname{tr}(a\mathbf S) = 0$.
    By the same calculations as in Theorem~\ref{thm:governing_equations_MDens}, but taking the constraint $\operatorname{tr}(aS)=0$ into account, we get the results.
\end{proof}


Note that for the balanced matrix transport we have the same alternative as with the scalar one, namely one can use only one parameter, the total mass, to normalize the density, as in \emph{simple unbalanced transport} \cite{Khesin2024}. 
Above, for $\mathrm{PGL}(k)$-currents, we use the pointwise normalization.

\begin{remark}
    The existence and uniqueness theorem for the matrix balanced transport is proved in the same way as Theorem \ref{thm:matrix-exist} covering the unbalanced setting.
\end{remark}
\begin{remark}
 One can also consider a different action of diffeomorphisms on matrix-valued densities for the case $k=n$.
 Namely, if matrix valued densities are defined on the tangent space to $M$ then one can regard elements
 of  ${\rm MDens} := \Omega^n(M, Sym_+(n))$ as metric-valued densities. The action of diffeomorphisms in that case will be not only by multiplication by the Jacobian, but via the action of the change of coordinates on the metric on $M$. The corresponding metric  will involve the Ebin metric on densities and differ from the one described above. This setting is carefully studied in \cite{BauerInPreparation}.
\end{remark}

\medskip




\section{Further ramifications of  transport}\label{sect:questions}


\begin{figure}[ht!]
    \centering
    \textbf{Balanced} \\[2ex]
    \begin{tikzcd}[column sep=normal, row sep=large]
        \mathrm{Diff}(M)\ltimes C^\infty(M,\mathrm{SO}(k)) \arrow[rr, hook]\arrow[d, "{(\varphi, A)\mapsto A(\varphi_*{\mathbf v})}"] 
        & 
        & \mathrm{Diff}(M)\ltimes C^\infty(M, \mathrm{PGL}(k))\arrow[d, "{(\varphi,A)\mapsto A(\varphi_*\mathbf S)A^\top}"] 
        \\
        \mathrm{VHProb} \arrow[rr, "\mathbf v \mapsto \mathbf v \mathbf v^\top"]\arrow[dr, "{\mathbf v \mapsto \lVert \mathbf v \rVert^2}"'] 
        & 
        & \mathrm{MProb} \arrow[dl, "{\mathbf S \mapsto \operatorname{tr}(\mathbf S)}"]
        \\
        & \mathrm{Prob}(M) & 
    \end{tikzcd}
    \\[4ex]
    \textbf{Unbalanced} \\[2ex]
    \begin{tikzcd}[column sep=normal, row sep=large]
        \mathrm{Diff}(M)\ltimes C^\infty(M,\mathrm{Conf}(k)) \arrow[rr, hook]\arrow[d, "{(\varphi, A)\mapsto A(\varphi_*{\mathbf v})}"] 
        & 
        & \mathrm{Diff}(M)\ltimes C^\infty(M, \mathrm{GL}(k))\arrow[d, "{(\varphi,A)\mapsto A(\varphi_*\mathbf S)A^\top}"] 
        \\
        \mathrm{VHDens} \arrow[rr, "\mathbf v \mapsto \mathbf v \mathbf v^\top"]\arrow[dr, "{\mathbf v \mapsto \lVert \mathbf v \rVert^2}"'] 
        & 
        & \mathrm{MDens} \arrow[dl, "{\mathbf S \mapsto \operatorname{tr}(\mathbf S)}"]
        \\
        & \mathrm{Dens}(M) & 
    \end{tikzcd}
    \caption{Commutative diagram for the relation between spaces for scalar, vector, and matrix valued optimal transport, in both balanced (upper) and unbalanced (lower) settings.
    (More precisely, the map from vector densities to matrix ones takes the former to  {\it positive semi-definite} 
    matrices ${\rm Sym}_+(k)$, and hence to the closure of $\rm MDens$, as discussed in Remark \ref{rem:closure}.)}
    \label{fig:universal_diagram}
\end{figure}
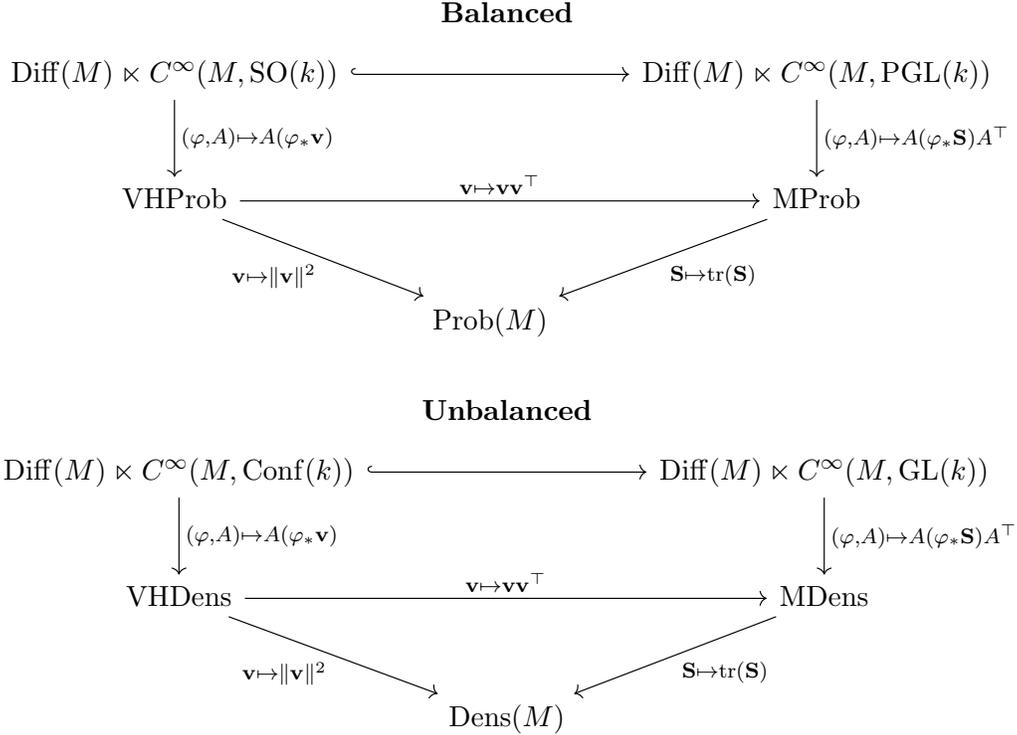


The relations between different parts of this paper are summarized in Figure~\ref{fig:universal_diagram}.
We continue here with several further directions which, in our opinion, would be interesting to explore.

\phantom{hj}
\begin{enumerate}

\item Do there exist interesting finite-dimensional models of the vector optimal transport, extending the action of linear transformations on Gaussians, cf.~\cite{Klas2017}? 

\item  By changing the discrete index (the density coordinates) to a continuous parameter, the velocity $v\in C^\infty(M, \mathbb R^k)$ changes to $v\in C^\infty(M\times \mathbb R, \mathbb R)$, and this covers a new variety of the equations, with only cosmetic changes.
It would be interesting to develop this setting and its applications, as it might be related to generalized flows as suggested by Brenier~\cite{Br1999}.

\item Adding a potential $U$ to the Lagrangian we extend the domain of applicability of transports of vector densities.
Namely, they are now governed not by geodesics, but by the Newton equation. 
The potential can encode various constraints and obstructions for the transport.

\item The groups considered above were related to densities and gauge transformations in a trivial bundle over  an orientable manifold. A similar optimal transport theory should exist for a topologically non-trivial 
bundle, as well as for densities over a non-orientable manifold.

\item It would be intersting to see if the relativistic Burgers described by Brenier~\cite{Br2022} could be recovered as a particular case of the 
 matrix Burgers equation (or, rather, of a compressible fluid equation) described above.  

\item Consider the Kähler setting, with vector half-densities over complex scalars, which gives a generalized Madelung transform via Fusca's map \cite{Fusca2017}.
It would be interesting to unify those two approaches in matrix transport and K\"ahler vector transport. 
They present essentially the same group action in two similar ways, in the real and complex directions. 
A natural question arises whether there is an analogue of Fusca's momentum map for the unbalanced mass transport. And if so, what are its Riemannian submersion and K\"ahler properties?
\end{enumerate}
\medskip

{\bf More on motivation} \,
In the introduction we mentioned  applications of optimal transport of vector and matrix densities to colored medical imaging and quantum spin liquids in physics. 
One more motivation  comes from 
the spectral analysis of multi-variable time series (e.g., \cite{Chen2018}). 
The related spectral content is assessed based on
certain statistics related to moments of the corresponding power spectral density. 
It may drift across frequencies over time as well as across principle directions (e.g., echo locating a moving target using antenna arrays), as captured by vector-valued time series. 
Therefore, in the transport between matrix-valued densities, we consider both 
the cost of shifting power across frequencies and the cost of rotating the corresponding principle axes. 

Another application is related to problems of optimal distribution in economics. Consider for instance the task of optimal distribution of vacation packages by wholesalers: assuming that one can change the length of stays at several resorts between a few standard ones, at a cost (i.e., the personnel's salary) as compared to the set of various packages originally bought by the wholesaler in bulk. Minimization of the corresponding cost becomes a problem of optimal transport of vector densities, where the coordinates are packages of a given length and at a given place. The metrics proposed in this paper are straightforward analogs of the Wasserstein metric, being based on the differential geometry of the problem.





\begin{appendices}

\section{Semi-invariant Hamiltonians for Lie groups}\label{sec:momentum_map_appendix}

The purpose of this appendix is to demonstrate, in an abstract setting, how the momentum map in symplectic geometry naturally appear in the reduced governing equations for Hamiltonian systems on the cotangent bundle of a Lie group.
We start with the definition of the momentum map.

\begin{definition}\label{def:momentum_map}
    Let $(M,\omega)$ be a symplectic manifold and let $G$ be a Lie group acting on $M$ by symplectic diffeomorphisms.     The action is called \emph{Hamiltonian} if there exists a \emph{momentum map}
    $\Phi\colon M\to \mathfrak{g}^*$ defined as follows. Assume that 
    the infinitesimal action for the Lie algebra element $\xi\in\mathfrak{g}$ is a Hamiltonian vector field, i.e. $x\mapsto \xi\cdot x \in T_xM$ is given by a Hamiltonian vector field $X_{H_\xi}$ for some Hamiltonian $H_\xi$ which depends linearly on $\xi$.
    Then the existence of the momentum map  $\Phi$ means a consistent choice of Hamiltonians so that
    \begin{equation*}
        \langle \Phi(x), \xi \rangle := H_\xi(x)
    \end{equation*}
    and $\Phi\colon M\to \mathfrak{g}^*$ is equivariant with respect to the coadjoint action on $\mathfrak{g}^*$.
\end{definition}

The result below, which we use in specific settings throughout the paper, is concerned with the special case of cotangent bundles $M=T^*Q$ with a cotangent lifted action of $G$ originating from an action of $G$ on $Q$.
We recall that cotangent lifted actions are always Hamiltonian, and that the momentum map $\Phi\colon T^*Q \to \mathfrak{g}^*$ in this case is defined by
\begin{equation*}
    \langle \Phi(q, p), \xi \rangle = \langle p,\xi\cdot q\rangle . 
\end{equation*}

\begin{proposition}\label{prop:poisson_reduction_abstract}
    Consider a Lie group $G$ acting on a space $S$ from the left.
    Take a Hamiltonian on $T^*G$ of the form $\bar H(g,m_g) = H(m, g\cdot s_0)$, where $m_g = m \cdot g$, $s_0\in S$, and $s = g\cdot s_0$.
    Then the Poisson reduced motion is given by
    \[
        \dot m + \operatorname{ad}^*_{\xi}m = -\Phi(s,\frac{\delta H}{\delta s}),\qquad \dot s + \xi\cdot s = 0,\qquad \xi = \frac{\delta H}{\delta m}
    \]
    where $\Phi\colon T^*S \to \mathfrak{g}^*$ is the momentum map for the cotangent lifted action of $G$ on $T^*S$.
\end{proposition}

\begin{proof}
    The canonical form of the equations are
    \begin{align*}
        \dot g &= \frac{\delta \bar H}{\delta m_g}, \\
        \dot m_g &= -\frac{\delta \bar H}{\delta g}.
    \end{align*}
    By standard calculations we have 
    \begin{equation*}
        \dot m_g = (\dot m + \operatorname{ad}^*_\xi m)\cdot g \qquad\text{and}\qquad 
        \dot g = \frac{\delta H}{\delta m} g .
    \end{equation*}
    Let $g_\epsilon = \operatorname{exp}(\epsilon\xi)g$ be a variation. 
    Then 
    \begin{equation*}
        \frac{d}{d\epsilon} H(m,g_\epsilon\cdot s_0) = \langle \frac{\delta H}{\delta s}, \xi\cdot s \rangle = \langle \Phi(s,\frac{\delta H}{\delta s}),\xi\rangle
        = \langle \Phi(s,\frac{\delta H}{\delta s})\cdot g,\xi\cdot g\rangle
    \end{equation*}
    Thus, 
    \begin{equation*}
        \frac{\delta \bar H}{\delta g} = \Phi(s,\frac{\delta H}{\delta s})\cdot g \,,
    \end{equation*}
    which proves the result.
\end{proof}


\section{An alternative metric on matrix densities}\label{sect:alternative}

In addition to \emph{extending} optimal transport from scalar to matrix valued densities, there is another approach to matrices by \emph{restricting} optimal transport on $\mathbb{R}^n$ to multivariate Gaussian densities (represented by positive definite matrices) and linear transport maps (cf.~\cite{Klas2017} and references therein).
As was suggested by F.-X. Vialard \cite{Via}, this approach provides the following alternative to the Bures-type  $L^2$-metric \eqref{eq:metric_Bures_MDens} on the group $\widehat G=  \operatorname{Diff}(M)\ltimes C^\infty(M,\operatorname{GL}(k))\ni (\varphi,A)$, which is consistent with its finite-dimensional analog:
\begin{equation}\label{eq:metric_gaussian}
    \begin{aligned}
        & \langle (\dot\varphi,\dot A), (\dot\varphi,\dot A)\rangle_{(\varphi,A)}
        = \int_M \left(|\dot\varphi\circ\varphi^{-1}|^2 \operatorname{tr}(\underbrace{A\varphi_*\mathbf{S}_0A^\top}_{\mathbf S})
        + \operatorname{tr}(\dot A \varphi_*\mathbf{S}_0 \dot A^\top)
        \right) \\
        & = \int_M \left(|u|^2 \operatorname{tr}(\mathbf{S})
        + \operatorname{tr}(\dot A A^{-1}\underbrace{A\varphi_*\mathbf{S}_0A^\top}_{\mathbf S} (\dot A A^{-1})^\top)
        \right) \\
        & = \int_M \left(|u|^2 \operatorname{tr}(\mathbf{S})
        + \operatorname{tr}(a \mathbf{S} a^\top)
        \right)\,,
    \end{aligned}
\end{equation}
where $(u,a)\in \mathfrak{X}(M)\ltimes C^\infty(M,\mathfrak{gl}(k))$ are the reduced coordinates, $u = \dot\varphi\circ\varphi^{-1}$ and $a = \dot A A^{-1}$, while $\mathbf S_0 \in \mathrm{MDens}$ is a  fixed initial matrix density.



For this metric, the Legendre transform under the $L^2$-Frobenius pairing is given by
\begin{equation}
    (u,a) \mapsto \left( u^\flat\otimes \operatorname{tr}(\mathbf S), [a\mathbf S] \right).
\end{equation}
In turn, this gives the Hamiltonian on $\widehat{\mathfrak{g}}^*\times \mathrm{MDens}$ as

\begin{equation}\label{eq:Hamiltonian_gaussian}
    \widehat{H}(u^\flat\otimes \operatorname{tr}(\mathbf S), [a\mathbf S], \mathbf S) = \frac{1}{2}\int_M \left(\frac{|u^\flat\otimes\operatorname{tr}(\mathbf S)|^2 }{\operatorname{tr}(\mathbf{S})}
        + \operatorname{tr}( (a \mathbf{S}) \mathbf{S}^{-1}(a \mathbf{S})^\top)
        \right).
\end{equation}

\begin{proposition}\label{thm:governing_equations_gaussian}   
    The Hamiltonian equations of motion, corresponding to the reduced geodesic equations on $\widehat{G}$, are
    \begin{align*}
        &
        \dot u + \nabla_u u = 2\frac{\operatorname{tr}b}{\varrho}, 
        \\ &
        \dot b + \nabla_u b = |u|^2\mathbf S + a^\top b
        \\ & \dot{\mathbf S} + L_u \mathbf S = b + b^\top,
    \end{align*}
    where, as before, we use the decomposition $\mathbf S = S\otimes\varrho = S \otimes \operatorname{tr}\mathbf S$, and where $b = a\mathbf S$.
\end{proposition}


Note that if $a$ is skew-symmetric then all the terms in the right-hand side of the equation for $\dot u$ vanish. This reflects that the alterative metric~\eqref{eq:metric_gaussian} on $\mathrm{MDens}$ descends to the vector half-density metric~\eqref{eq:Bures_metric_vectors}.

\begin{proof}[Proof sketch]
    The variational derivatives of the Hamiltonian are
    \begin{multline*}
        \frac{\delta \widehat H}{\delta (u^\flat\otimes\operatorname{tr}(\mathbf S))} = u, \quad 
        \frac{\delta \widehat H}{\delta (a\mathbf S)} = a, \\ 
        \frac{\delta \widehat H}{\delta \mathbf{S}} = -\frac{1}{2}\big( \lvert u \rvert^2 I + \operatorname{tr}(\mathbf S^{-1} (a\mathbf S)^\top (a\mathbf{S})\mathbf S^{-1}\delta\mathbf S )\big) =  -\frac{1}{2}\big( \lvert u \rvert^2 I + a^\top a \big)\,,
    \end{multline*}
    where $I \in T^*_{\mathbf{S}}\mathrm{MDens}$ is the identity matrix. 
    The general form of the corresponding Hamiltonian equations is obtained in Proposition~\ref{prop:poisson_reduction_abstract}. Together with  Lemma~\ref{lem:ad_star} and Lemma~\ref{lem:momentum_map_matrix_case} this yields the equations as stated.
\end{proof}

\begin{proof}
    The variational derivatives of the Hamiltonian are
    \begin{multline*}
        \frac{\delta \widehat H}{\delta (u^\flat\otimes\operatorname{tr}(\mathbf S))} = u, \quad 
        \frac{\delta \widehat H}{\delta (a\mathbf S)} = a, \\ 
        \frac{\delta \widehat H}{\delta \mathbf{S}} = -\frac{1}{2}\big( \lvert u \rvert^2 I + \operatorname{tr}(\mathbf S^{-1} (a\mathbf S)^\top (a\mathbf{S})\mathbf S^{-1}\delta\mathbf S )\big) =  -\frac{1}{2}\big( \lvert u \rvert^2 I + a^\top a \big)\,,
    \end{multline*}
    where $I \in T^*_{\mathbf{S}}\mathrm{MDens}$ is the identity matrix. 
    Next, the general form of the corresponding Hamiltonian equations is obtained in Proposition~\ref{prop:poisson_reduction_abstract} in Appendix, namely
    \begin{equation}\label{eq:ham_eq_abstract_gaussian}
        \frac{d}{dt}(u^\flat\otimes\operatorname{tr}(\mathbf S), a\mathbf S) = \operatorname{ad}^*_{(\frac{\delta H}{\delta u^\flat\otimes\operatorname{tr}(\mathbf S)},\frac{\delta H}{\delta a \mathbf S})}(u^\flat\otimes\operatorname{tr}(\mathbf S), a\mathbf S) - \Phi(\mathbf{S}, \frac{\delta H}{\delta \mathbf{S}}).
    \end{equation}
    Then from Lemma~\ref{lem:momentum_map_matrix_case} we have
    \begin{align*}
        &\Phi(\mathbf{S}, \frac{\delta H}{\delta \mathbf{S}}) =
        \Big( -\frac{1}{2}\operatorname{tr}\Big(\mathbf S d \big( \lvert u \rvert^2 + a^\top a \big) \Big), -\big( \lvert u \rvert^2 +  a^\top a\big)\mathbf S \Big) 
        \\ 
        &= \Big( - \frac{1}{2}d \lvert u \rvert^2\otimes \operatorname{tr}(\mathbf S) -\frac{1}{2}\operatorname{tr}\big(a\mathbf S d a^\top + \mathbf S a^\top d a  \big), -\big( \lvert u \rvert^2 +  a^\top a\big)\mathbf S \Big) 
        \\ 
        &= \Big( - \frac{1}{2}d \lvert u \rvert^2\otimes \operatorname{tr}(\mathbf S) -\frac{1}{2}\operatorname{tr}\big(a\mathbf S d a^\top + d a^\top a \mathbf S \big), -\big( \lvert u \rvert^2 +  a^\top a\big)\mathbf S \Big) 
        \\ 
        &= \Big( - \frac{1}{2}d \lvert u \rvert^2\otimes \operatorname{tr}(\mathbf S) -\operatorname{tr}\big(a\mathbf S d a^\top \big), -\big( \lvert u \rvert^2 +  a^\top a\big)\mathbf S \Big) ,
    \end{align*}
    and a calculation as in Lemma~\ref{lem:ad_star} shows that 
    \begin{equation*}
        \operatorname{ad}^*_{(u,a)}(\boldsymbol{m}, \boldsymbol{\beta}) = \Big(
            -L_u (\boldsymbol{m}) - \operatorname{tr}(\boldsymbol{\beta} d a^\top),
            -L_u (\boldsymbol{\beta})
         \Big).
    \end{equation*}
    Thus,
    \begin{align*}
        &\operatorname{ad}^*_{(\frac{\delta H}{\delta u^\flat\otimes \operatorname{tr}(\mathbf S)}, \frac{\delta H}{\delta a\otimes\operatorname{tr}(\mathbf S)})}(u^\flat\otimes\operatorname{tr}(\mathbf S),a\mathbf S) =  \big( -L_u (u^\flat\otimes\operatorname{tr}(\mathbf S)) - \operatorname{tr}(a\mathbf S \, d a^\top), -L_u (a\mathbf S) \big). 
    \end{align*}
    From equation \eqref{eq:ham_eq_abstract_gaussian} we then get
    \begin{align*}
        \frac{d}{dt}(u^\flat\otimes\operatorname{tr}(\mathbf S)) &= -L_u(u^\flat\otimes\operatorname{tr}(\mathbf S)) + \frac{1}{2}d|u|^2 \otimes\operatorname{tr}(\mathbf S)
        -\frac{1}{2}\operatorname{tr}\big(\underbrace{(a\mathbf S - a\mathbf S )}_0 d a^\top  \big)
        .        
    \end{align*}
    Likewise, for the $a\mathbf S$ component we obtain
    \begin{align*}
        \frac{d}{dt}(a\mathbf S) = -L_u(a\mathbf S) + (|u|^2 + a^\top a)\mathbf S . 
    \end{align*}
    Furthermore, the transported variables are
    \begin{equation*}
        \dot{\mathbf S} = -L_u \mathbf S + a\mathbf S + \mathbf S a^\top.
    \end{equation*}
    Thus, the trace density evolves as
    \begin{equation*}
        \frac{d}{dt}\operatorname{tr}(\mathbf S) = -L_u \operatorname{tr}(\mathbf S) + 2\operatorname{tr}(a\mathbf S).
    \end{equation*}
    Moreover, the evolution of the vector field $u$ is given by
    \begin{equation*}
        \dot u + \nabla_u u = 2 \frac{\operatorname{tr}(a \mathbf S) u}{\operatorname{tr}(\mathbf S)}. 
    \end{equation*}
    In the coordinates for matrix density $\mathbf S = S\otimes\varrho$ as before, this gives
    \begin{equation*}
        \dot u + \nabla_u u = 2\operatorname{tr}(aS). 
    \end{equation*}
    This proves the result.
\end{proof}



\end{appendices}

\bibliographystyle{abbrv}
\bibliography{references}

\end{document}